\documentclass[final]{siamltex}

\usepackage{latexsym, enumerate}
\usepackage{eepic}
\usepackage{epic}
\usepackage{graphicx}
\usepackage{color}
\usepackage{ifpdf}
\usepackage{amssymb,amsmath,epsfig, algorithm,algorithmicx,multirow}
\usepackage{amsfonts, dsfont}
\usepackage{subfigure}

\newcommand{\hmP}{\widehat{\mathrm{P}}}
\newcommand{\mP}{\mathrm{P}}
\newcommand{\mfu}{\mathbf{u}}

\title{A hybrid HDMR for mixed multiscale finite element methods with
       application to flows in random porous media
  \thanks{L. Jiang and J. D. Moulton acknowledge funding by the
    Department of Energy at Los Alamos National Laboratory under
    contracts DE-AC52-06NA25396 and the DOE Office of Science Advanced
    Computing Research (ASCR) program in Applied Mathematical
    Sciences.
}}

\author{Lijian Jiang\thanks{Applied Mathematics and Plasma Physics,
Los Alamos National Laboratory, NM 87545 ({\tt ljiang@lanl.gov}). Corresponding author.}
\and
J. David Moulton\thanks{Applied Mathematics and Plasma Physics,
Los Alamos National Laboratory, NM 87545 ({\tt moulton@lanl.gov}).}
\and
Jia Wei\thanks{CGGVeritas, Houston, TX 77072 ({\tt weijialily@gmail.com})}
 }

\begin{document}

\maketitle

\begin{abstract}
  Stochastic modeling has become a popular approach to quantify
  uncertainty in flows through heterogeneous porous media. In this
  approach the uncertainty in the heterogeneous structure of material
  properties is often parametrized by a high-dimensional random
  variable, leading to a family of deterministic models.  The
  numerical treatment of this stochastic model becomes very
  challenging as the dimension of the parameter space increases.  To
  efficiently tackle the high-dimensionality, we propose a hybrid
  high-dimensional model representation (HDMR) technique, through
  which the high-dimensional stochastic model is decomposed into a
  moderate-dimensional stochastic model in the most active random
  subspace, and a few one-dimensional stochastic models. The derived
  low-dimensional stochastic models are solved by incorporating the
  sparse-grid stochastic collocation method with the proposed hybrid
  HDMR.  In addition, the properties of porous media, such as
  permeability, often display heterogeneous structure across multiple
  spatial scales. To treat this heterogeneity we use a mixed
  multiscale finite element method (MMsFEM). To capture the non-local
  spatial features (i.e. channelized structures) of the porous media
  and the important effects of random variables, we can hierarchically
  incorporate the global information individually from each of the
  random parameters. This significantly enhances the accuracy of the
  multiscale simulation.  Thus, the synergy of the hybrid HDMR and the
  MMsFEM reduces the dimension of the flow model in both the
  stochastic and physical spaces, and hence, significantly decreases
  the computational complexity.  We analyze the proposed hybrid HDMR
  technique and the derived stochastic MMsFEM.  Numerical experiments
  are carried out for two-phase flows in random porous media to
  demonstrate the efficiency and accuracy of the proposed hybrid HDMR
  with MMsFEM.
\end{abstract}



\pagestyle{myheadings}

\thispagestyle{plain}
\markboth{L. Jiang, J. D. Moulton and J. Wei}{MMsFEM based on hybrid HDMR}

\section{Introduction}

The modeling of dynamic flow and transport processes in geologic
porous media plays a significant role in the management of natural
resources, such as oil reservoirs and water aquifers. These porous
media are often created by complex geological processes and may
contain materials with widely varying abilities to transmit
fluids. Thus, multiscale phenomena are inherent in these applications
and must be captured accurately in the model. In addition, due to
measurement errors and limited knowledge of the material properties
and external forcing, modeling of subsurface flow and transport often
uses random fields to represent model inputs (e.g.,
permeability). Then the system's behavior can be accurately predicted
by efficiently simulating the stochastic multiscale model.  The
existence of heterogeneity at multiple scales combined with this
uncertainty brings significant challenges to the development of
efficient algorithms for the simulation of the dynamic processes in
random porous media.  As a result, the interest in developing
stochastic multiscale methods for stochastic subsurface models has
steadily grown in recent years.

The existence of uncertainty in random porous media is an important
challenge for simulations. One way to describe the uncertainty is to
model the random porous media as a random field which satisfies
certain statistical correlations. This naturally results in describing
the flow and transport problem using stochastic partial differential
equations (SPDE). Over the last few decades several numerical methods
have been developed for solving SPDEs.  The existing stochastic
numerical approaches roughly fall into two classes: (1) non-intrusive
schemes and (2) intrusive schemes. In non-intrusive schemes, the
existing deterministic solvers are used without any modification to
solve a (large) set of deterministic problems, which correspond to a
set of samples from the random space. This leads to a set of outputs,
which are used to recover the desired statistical quantities.  Monte
Carlo methods \cite{fis96} and stochastic collocation methods fall
into this category.  In contrast, intrusive schemes treat the
approximation of the probabilistic dependence directly in the
SPDE. Thus, discretization leads to a coupled set of equations that
are fundamentally different from discretizations of the original
deterministic model, and hence, raises new challenges for nonlinear
solvers. Typical examples from this class are stochastic
Galerkin \cite{ghs91,epsu09,mk05} and perturbation methods
\cite{zlt00}. A broad survey of these methods can be found in
\cite{lk10}.  Among these methods, stochastic collocation methods have
gained significant attention in the research community because they
share the fast convergence property of the stochastic finite element
methods while having the decoupled nature of Monte Carlo methods. A
stochastic collocation method consists of two components: a set of
collocation points and an interpolation operator. The collocation
points are selected based on specific objectives, and choosing
different objectives or constraints leads to different methods.  Two
popular collocation methods are the full-tensor product method
\cite{bnt07} and the Smolyak sparse-grid method
\cite{bnr00,ntw08,sm63}. Sparse-grid stochastic collocation is known
to have the same asymptotic accuracy as full-tensor product
collocation, but uses significantly fewer collocation points.
Unfortunately, the number of collocation points required in the
sparse-grid method still increases dramatically for high-dimensional
stochastic problems. Hence, application of this method is still limited
to problems in a moderate-dimensional stochastic space.

Due to small correlation lengths in the covariance structure, the
uncertainty in characterizations of porous media are often
parametrized by a large number of random variables, e.g., using a
truncated Fourier type expansion for random fields.  Consequently, the
model's input is defined in a high-dimensional random parameter
space. Sampling in high-dimensional random spaces is computationally
demanding and very time-consuming. If the sampling of random space is
conducted in the full random space through the stochastic collocation
method, then the number of samples increases sharply with respect to
the dimension of the random space. This is the notorious {\it curse of
  dimensionality}, and is a fundamental problem facing stochastic
approximations in high-dimensional stochastic spaces.  Dimension
reduction techniques such as proper orthogonal decomposition,
principal component analysis, reduced basis methods \cite{rht08},
novel random field expansion techniques \cite{jp12, xu011,xu012}  and
high-dimensional model representations (HDMRs) \cite{rass99}, strive
to address this problem. Among these methods, the HDMR is one of the
most promising methods for efficient stochastic dimension reduction in
subsurface applications, and has received considerable attention in
recent years \cite{mz10, mz11, ng08, yclk11}.  The HDMR was originally
developed in the application of chemical models by \cite{rass99}, but
was later cast as a general set of quantitative assessment and
analysis tools for capturing the high-dimensional relationship between
model inputs and model outputs.  HDMR has been used for improving the
efficiency of deducing high-dimensional input-output system behavior,
and can be employed to reduce the computational effort.  In practice,
to avoid dealing with the full high-dimensional random space, a
truncated HDMR technique is used to decompose a high-dimensional model
into a set of low-dimensional models, each of which needs much less
computational effort.  The {\it curse of dimensionality} can be
suppressed significantly by using this approach to HDMR.  However,
some drawbacks exist in the traditional truncated HDMR
\cite{css75,rass99} and related methods such as the adaptive HDMR
\cite{mz10,yclk11}. For example, often too many low-dimensional models
need to be computed, and high-order cooperative effects from important
random variables may be neglected.  To overcome these drawbacks, we
propose a hybrid HDMR, which implicitly uses the complete HDMR in a
moderate-dimensional space and explicitly uses the first-order
truncated HDMR for the remaining dimensions.  Thus the hybrid HDMR
decomposes a high-dimensional model into a moderate-dimensional model
and a few one-dimensional models.  The moderate-dimensional space is
spanned by the most active random dimensions. We use sensitivity
analysis to identify the most active random dimensions. Then each
low-dimensional stochastic model in the hybrid HDMR is solved by using
sparse-grid stochastic collocation method. The proposed hybrid HDMR
renders a good approximation in stochastic space and substantially
improves the efficiency for high-dimensional stochastic models.
Therefore, a good trade-off between computational complexity and
dimension reduction error can be achieved with the hybrid HDMR technique.

Porous media often exhibit complex heterogeneous structures that are
inherently hierarchical and multiscale in nature, and hence, pose a
significant challenge for developing accurate and efficient numerical
methods. Simulating flow in porous media using a very fine grid to
resolve this heterogeneous structure is computationally very expensive
and possibly infeasible. However, disregarding the heterogeneities can
lead to large errors. Thus, many multiscale methods including the MMsFEM
\cite{ch03}, variational multiscale method \cite{hughes98}, two-scale
conservative subgrid approaches \cite{arbogast02} and heterogeneous
multiscale method \cite{ee03}, multiscale finite volume method
\cite{jennylt03}, spectral MsFEM \cite{egw11} (see \cite{eh09}
for more complete references), have been developed over the last few
decades to capture the influence of fine- and multi-scale
heterogeneities in under-resolved simulations.
We note that these multiscale methods share some similarities
\cite{eh09}.  In this paper, we focus on the MMsFEM.  The main idea
behind MMsFEM is to incorporate fine-scale information into the finite
element velocity basis functions, and hence, capture the influence of
these fine scales in a large-scale mixed finite element formulation.
The MMsFEM retains local conservation of velocity flux and has been
shown to be effective for solving flow and transport equations in
heterogeneous porous media.  In many cases, the overhead of computing
multiscale basis functions can be amortized as they can be computed
once for a particular medium, and then reused in subsequent
computations with different source terms and boundary conditions. This
leads to a large computational savings in simulating the flow and
transport process where the flow equation needs to be solved many
times dynamically.  When porous media exhibit non-separable scales,
some global information is needed to represent non-local effects
(e.g., channel, fracture and shale barriers) and is used to construct
the multiscale basis functions.  Using global information can
significantly improve the simulation accuracy in these important
geological settings \cite{aa,aej08,jem10,jml}.  If the global
information is incorporated into MMsFEM, we refer to the MMsFEM as
global MMsFEM (G-MMsFEM). In the paper, we extend the central concept
of the global information to MMsFEM based on the hybrid HDMR.

Combining different multiscale methods and stochastic numerical
approaches yields various stochastic multiscale methods
\cite{gz09,gyz11, gmp10,jml,jp12,mz11,xu07}. These methods
established a general approach to solve problems in multiscale
stochastic media.  Specifically, dimension reduction techniques are
applied, particularly focusing on the multiscale spatial dimension of
these stochastic models, to reduce the overall computational cost.
For example, within the stochastic Galerkin framework several earlier
works (e.g., \cite{ag08,ng08,rht08}) proposed solving an optimization
problem to select a set of optimal basis functions to define a reduced
model. Although, these approaches reduce the resolution of the spatial
dimension, the stochastic dimension is not reduced.  Therefore, the
computational cost of these approaches may still be too expensive in
high-dimensional stochastic spaces.

In this paper, we propose a hybrid HDMR framework, and use the
sparse-grid stochastic collocation method with MMsFEM to develop a
reduced multiscale stochastic model.  Sensitivity analysis is used
to control the reduction of the stochastic dimension.  For the MMsFEM,
we hierarchically utilize approximate global information that is
aligned with the terms of the hybrid HDMR.  Specifically, only the
required basis functions are computed and to amortize their
computational cost they are held constant in time.  Thus, the proposed
multiscale stochastic model reduction approach is able to reduce a
high-dimensional stochastic multiscale model in both stochastic space,
and the resolution of physical space.  Compared with traditional
truncated HDMR techniques, much better efficiency and very good
accuracy are achieved with this hybrid HDMR.  We analyze the proposed
multiscale stochastic model reduction approach and investigate its
application to flows in random and heterogeneous porous media.
Important statistical properties (e.g., mean and variance) of the
outputs of the stochastic flow models are computed and discussed.

The rest of the paper is organized as follows. In Section $2$, we
briefly introduce the background on the flow and transport models in
random porous media.  In Section $3$, we present a general framework
for HDMR and propose the hybrid HDMR technique. Some theoretical
results and computational complexity are also addressed in this
section.  Section $4$ is devoted to presenting MMsFEM integration with
the hybrid HDMR.  In Section $5$, numerical examples using two-phase
flow are presented to demonstrate the performance of the proposed the
hybrid HDMR with MMsFEM.  Finally, some conclusions and closing
remarks are made.

\section{Background and notations}

\subsection{Two-phase flow system and its stochastic parametrization}
Let $D$ be a convex bounded domain in $\mathbb{R}^d$ ($d=2, 3$) and $(\Omega,F, P)$ be a probability space, where $\Omega$ is the set of outcomes, $F$ is the $\sigma$-algebra
generated by $\Omega$, and $P$ is a probability measure.

%Let $k(x, \omega)$ be a random permeability field.
We consider two-phase flow and transport in a random permeability
field, $k(x,\omega)$. Here the two phases are referred to as water and
oil, and designated by subscripts $w$ and $o$, respectively.  The
equations of two-phase flow and transport (in the absence of gravity
and capillary effects) can be written:
\begin{eqnarray}
-div(\lambda(S) { k(x, \omega)} \nabla p)&=q, \label{pres-eqn}\\
{\partial S\over\partial t}+ div(uf_w(S))&=0, \label{sat-eqn}
\end{eqnarray}
where the total mobility $\lambda(S)$ is given by
$\lambda(S)=\lambda_w(S)+\lambda_o(S)$ and $q$ is a source term.
Here $\lambda_w(S)=k_{rw}(S)/\mu_w$ and  $\lambda_o(S)=k_{ro}(S)/\mu_o$,
where $\mu_o$ and  $\mu_w$ are viscosities of oil and water phases, respectively,
and $k_{rw}(S)$ and $k_{ro}(S)$ are relative permeabilities of oil and water phases, respectively. Here
$f_w(S)$ is the fractional flow of water and  given by $f_w=\lambda_w /
(\lambda_w+\lambda_o)$,
Equation (\ref{pres-eqn}) is the flow equation governing the water
pressure, and (\ref{sat-eqn}) is the transport (or saturation) equation.
According to Darcy's law, the total velocity $u$ in  (\ref{sat-eqn}) is given  by
\begin{equation}
u =u_w + u_o=-\lambda(S) { k } \nabla p. \label{tv-eqn}
\end{equation}

A random field may be parameterized by a Fourier type expansion, such as the expansion of Karhunen-Lo\`{e}ve (ref.\cite{Lo77}), polynomial chaos or wavelet \cite{lk10}.
This often gives rise to an infinite-dimensional random space. For computation, we truncate such an expansion to approximate the random field. Then $k(x, \omega)$ can be formally described by
\begin{equation}
\label{approx-random}
k(x, \omega)\approx k\big(x, \theta_1(\omega), \theta_2(\omega),\dots, \theta_N(\omega)\big).
\end{equation}
For example, if a random field is characterized by a covariance
structure, then the random field can be approximated by a finite sum
of uncorrelated random variables through a truncated Karhunen-Lo\`{e}ve
expansion (KLE).  To obtain an accurate approximation, a large number
$N$ of random parameters is required in (\ref{approx-random}). This
leads to a family of deterministic models in a high-dimensional random
parameter space.

To simplify the presentation, we make the following assumption,
\[
k(x, \omega)=k(x, \Theta), \quad \text{ where $\Theta:=\Theta(\omega)=\big(\theta(\omega_1),\theta_2(\omega),\dots, \theta_N(\omega)\big)\in \mathbb{R}^N$}.
\]
Let $I^N\subset \mathbb{R}^N$ be  the image of $\Theta$, i.e., $I^N=\Theta(\Omega)$, and $\rho(\Theta)=\Pi_{i=1}^N \rho_i(\theta_i)$ be the joint probability function of $(\theta_1,\cdots, \theta_N)$.
By equations (\ref{pres-eqn}),  (\ref{sat-eqn}),  (\ref{tv-eqn}) and parametrization of permeability field,  we formulate the following stochastic two-phase flow system:  Find random fields
$p(x, \Theta, t)$, $v(x, \Theta, t)$,  $S(x,\Theta,t):  D \times I^N \times (0,T] \longrightarrow \mathbb{R}$ such that they almost surely (a.s)
satisfy the following equations subject to initial and boundary conditions,
\begin{eqnarray}
\label{tp-system}
\begin{split}
\begin{cases}
&\mbox{div}(u) = q \\
& u=-\lambda(S) k(x,\Theta )\nabla p\\
&\displaystyle \frac{\partial S}{\partial t}+\mbox{div}\big(uf_w(S)\big) =0.\\
\end{cases}
\end{split}
\end{eqnarray}
To further simply the notation, we will suppress the spatial variable $x$ and temporal variable $t$ in the rest of paper when no ambiguity occurs.

\subsection{Sparse grid stochastic collocation method}
\label{sect-sgc}
 For stochastic two-phase flow systems \eqref{tp-system}, the statistical properties (e.g., mean and variance)  of solutions are   of  interest. These properties may be obtained by first sampling
  the parameter random space using, for example, a Monte Carlo method or sparse grid collocation method, then solving the deterministic problems for the samples and analyzing the corresponding results to obtain the desired statistical quantities. The convergence of Monte Carlo methods is slow and a very large   number of samples may be  required,  which leads to high computational cost. Instead, we use the Smolyak sparse-grid collocation method \cite{sm63}, where the Smolyak interpolant  $\mathcal{A}(N+\ell,N)$ ($\ell\geq 1$) is a linear combination of tensor product interpolants with the property: only products with a relatively small number of nodes are used and the linear combination is chosen in such a way that an interpolation property for $N=1$ is preserved for $N >1$ \cite{ bnr00}.
In the notation $\mathcal{A}(N+\ell,N)$, the $\ell$ represents the interpolation level.  The sparse-grid collocation method is known to have the same asymptotic accuracy as tensor product collocation method, while requiring many  fewer collocation points as the parameter dimension increases \cite{ntw08}.

Sparse grids have been successfully applied  to stochastic
collocation in many recent works (e.g.,   \cite{mz10, mz11, jml}).
Based on Smolyak formula (ref.\cite{bnr00}), a set of collocation points $\{\Theta^{(j)}\}_{j=1}^{N_c}$ in $I^N$  are specially chosen, where $N_c$ is the number of collocation points.
With these chosen collocation points and corresponding weights $\{w^{(j)}\}_{j=1}^{N_c}$, the statistical properties of the solutions can be obtained.  At  each of the collocation points, the deterministic system  \eqref{tp-system}   is solved and the output, for example,  $S(x, \Theta^{(j)}, t)$ is obtained. Then the mean  of $S(x, \Theta, t)$ can be estimated by
\begin{eqnarray*}
E[S(x,\Theta, t)] &=& \int_{I^N}S(x,\Theta, t)\rho(\Theta)d\Theta \\
    &\approx & \int_{I^N}\mathcal{A}(N+\ell, N) \big[S(x,\Theta, t)\big]\rho(\Theta)d\Theta= \sum_{j=1}^{N_c}S(x,\Theta^{(j)}, t)w^{(j)}.
\end{eqnarray*}
Here the weights  $\{w^{(j)}\}_{j=1}^{N_c}$ are determined by the basis functions of $\mathcal{A}(N+\ell, N)$ and joint probability function $\rho(\Theta)$ (ref.\cite{gg98}).
Similarly, the variance of $S(x, \Theta^{(j)}, t)$ can be obtained by
\begin{eqnarray}
\label{comp-var}
\begin{split}
\mbox{Var}[S(x,\Theta, t)]&=&\int_{I^N}\big(S(x,\Theta, t)-E[S(x,\Theta, t)]\big)^2 \rho(\Theta)d\Theta \\
&\approx& \sum_{j=1}^{N_c}S^2(x,\Theta^{(j)}, t)w^{(j)}-E^2[S(x,\Theta, t)].
\end{split}
\end{eqnarray}

Let $H(N+\ell,  N)$ denote the number of  collocation points for Smolyak sparse grid interpolation $\mathcal{A}(N+\ell,N)$. Then it follows that (see \cite{bnr00})
\begin{equation}
\label{coll-nodes}
H(N+\ell, N)\approx \frac{2^{\ell}}{\ell!}N^{\ell} \quad \text{for $N\gg  1$}.
\end{equation}
This implies that  the  number of  collocation points algebraically increases with respect to the dimension $N$.
We utilize  Smolyak sparse grid collocation for the numerical computation.
The stochastic approximation of the Smolyak sparse grid collocation method  depends on the
number of  collocation points and the dimension $N$ of the random parameter space.
The convergence analysis in \cite{ntw08} implies that the convergence  of Smolyak sparse grid collocation is exponential with respect to the number of Smolyak points, but depends on the parameter dimension $N$. This exponential convergence rate behaves algebraically for $N\gg1$.

\section{High dimensional model representation}

The truncated KLE leads to the family of two-phase flow deterministic
models (\ref{tp-system}) with a high-dimensional parameter $\Theta$.
The most challenging part of solving such a high-dimensional
stochastic system is to discretize the high-dimensional stochastic
space.  There exist a few methods for the discretization of the random
space \cite{ag08, bnr00, bnt07,epsu09, fis96,gyz11,ghs91,lk10, mz11,ng08, ntw08, xiu05, zlt00} (more references can be found therein). Among these methods,
the sparse-grid collocation method has been widely used and generates
completely decoupled systems, each of which is the same size as the
deterministic system.  This method is usually very efficient in
moderate-dimensional spaces. However, when the dimension of the random
parameter is large, a large number of collocation points are required
(\ref{coll-nodes}), and the deterministic model (\ref{tp-system}) must
be solved at each of these collocation points.  Consequently, the
efficiency of this collocation method will deteriorate in a
high-dimensional space.  To overcome this difficulty, we use a
high-dimensional model representation (HDMR) to reduce the stochastic
dimension and enhance the efficiency of the simulation.  By truncating
HDMR, the high-dimensional model can be decomposed into a set of
low-dimensional models.  Hence the computational effort can be
significantly reduced \cite{mz10,yclk11}.

In the following section, we present a general HDMR framework and
propose a hybrid HDMR method.

\subsection{A general HDMR framework}
\label{Sec:abstract-HDMR}

In this section, we adopt the decomposition of multivariate functions
described in $\cite{ksww09}$ to present a general HDMR framework in
terms of operator theory.  Let $\mathcal{F}$ be a linear space of
real-valued functions $f(\Theta)$ defined on a cube $I^N$ and
$\Theta:=(\theta_1, \theta_2, \cdots, \theta_N)\in I^N$. In this
discussion $f(\Theta)$ represents a relationship between the random
vector model input, $\Theta$, and a model output (e.g., saturation,
water-cut).

To introduce HDMR approach, we define a set of projection operators as follows.

\begin{definition}\cite{ksww09}
  Let $\{\mathrm{P_j}\}_{j=1}^N$ be a set of commuting projection operators on $\mathcal{F}$ satisfying the following property:
\begin{equation}
\label{project-assum}
\mathrm{P_j}(f)=f \quad \text{if $f$ does not  depend on $\theta_j$,  and $\mathrm{P_j}(f)$ does not depend on $\theta_j$}.
\end{equation}
\end{definition}

Let $\mathbf{u}\subseteq \{1,\cdots, N\}$ be a subset and $I$ the identity operator, we define
\[
\mathrm{P}_{\mathbf{u}}:=\Pi_{j\in \mathbf{u}} \mathrm{P_j}, \quad \text{and}  \quad \mathrm{P}_{\emptyset}:=I.
\]
Associated with projection $\mathrm{P}_{\mathbf{u}}$, we define $\hmP_{\mfu}:=\mathrm{P}_{\{1,\cdots, N\}\setminus \mfu}$. Then $\hmP_{\mfu} f$ only depends
on the variables with indices in $\mfu$.

 Let $\mu$ be a measure on Borel subsets of $I^N$ and a product measure with unit mass, i.e.,
\begin{equation}
\label{def-mu}
d\mu(\Theta):=\Pi_{i=1}^N d\mu_i(\theta_i), \quad  \int_{I} d\mu_i(\theta_i)=1, \quad i=1,\cdots, N.
\end{equation}
The inner product $\langle \cdot, \cdot \rangle$ on $\mathcal{F}$
induced by the measure $\mu$ is defined as follows:
\[
\langle f, h\rangle :=\int_{I^N} f(\Theta) h(\Theta) d\mu(\Theta),  \quad  f, g\in \mathcal{F}.
\]
The norm $\|\cdot\|_{\mathcal{F}}$ on $\mathcal{F}$ is defined by $\|f\|_{\mathcal{F}}:=\langle f, f\rangle ^{1/2}$.
Given the measure $\mu$, we can specify the projection operator in (\ref{project-assum}).
We assume that the functions in $\mathcal{F}$ are integrable with respect to  $\mu$. For any $f\in \mathcal{F}$ and $j\in \{1, \cdots, N\}$, we define
\begin{equation}
\label{def-p}
\mP_j f(\Theta)=\int_I f(\Theta) d\mu_j(\theta_j).
\end{equation}
 Consequently, $\hmP_{\mfu} f (\Theta)=\int_{I^{N-|\mfu|}} f(\Theta)\Pi_{j\notin \mfu} d\mu_j(\theta_j)$ for any $\mfu\subseteq \{1, \cdots, N\}$.
 By the definitions of these projection operator, we can obtain  a decomposition of $\mathcal{F}$ as following (see \cite{ra99}),
\begin{equation}
\label{Decomp-F}
\mathcal F=\oplus_{\mfu\subseteq\{1,\cdots, N\}} \mathcal{F}_{\mfu}, \quad \text{where} \quad \mathcal{F}_{\mfu}= \bigg( \Pi_{j\in \mathbf{u}} (I-\mathrm{P}_j)\bigg)\hmP_{\mfu} \mathcal{F}.
\end{equation}

For any $\mfu \subseteq\{1,\cdots, N\}$, we define
\[
\mathrm{M}_{\mfu}=\bigg( \Pi_{j\in \mathbf{u}} (I-\mathrm{P}_j)\bigg)\hmP_{\mfu}.
\]
Then we can show the operators $\mathrm{M}_{\mfu}$ are commutative projection operators and mutually orthogonal, i.e.,
\[
\mathrm{M}_{\mfu}^2=\mathrm{M}_{\mfu},  \quad \mathrm{M}_{\mfu}\mathrm{M}_{\mathbf{v}}=0 \quad \text{for $\mfu\neq \mathbf{v}$}.
\]
The equation  (\ref{Decomp-F}) gives an abstract  HDMR expansion of $f$  by
\begin{equation}
\label{HDMR-1}
f=\sum_{\mfu\subseteq\{1, \cdots, N\}} f_{\mfu}=\sum_{\mfu\subseteq\{1, \cdots, N\}} \mathrm{M}_{\mfu}f,    \quad \text{where $\mathrm{M}_{\mfu}f\in \mathcal{F}_{\mfu}$}.
\end{equation}
The equation  (\ref{Decomp-F}) also implies that
\[
\|f\|^2_{\mathcal{F}}=\sum_{\mfu\subseteq\{1, \cdots, N\}} \|\mathrm{M}_{\mfu}f\|_{\mathcal{F}}^2, \quad  \hmP_{\emptyset}(\mathrm{M}_{\mfu}f)=0 \quad \text{for $\mfu\neq \emptyset$}.
\]
 Any set of commutative projectors $\{\mathrm{M}_{\mathbf{v}}\}$   generate a distributive lattice whose elements are obtained by all possible combinations (addition and multiplication) of the projectors in the set. The lattice
has a unique maximal projection  operator $\mathcal{M}$,   which gives the algebraically best approximation to the  functions in $\mathcal{F}$ \cite{gordon69}. The range of the maximal projection operator $\mathcal{M}$ for the lattice is the union of the ranges of $\{\mathrm{M}_{\mathbf{v}}\}$.  Because the commutative projectors $\{\mathrm{M}_{\mathbf{v}}\}$ are mutually orthogonal here, the
maximal projection operator $\mathcal{M}$ and the range $\mathcal{F}_{\mathcal{M}}$ of $\mathcal{M}$ have explicit expressions as follows:
\[
\mathcal{M}=\sum_{\mathbf{v}}\mathrm{M}_{\mathbf{v}},  \quad  \mathcal{F}_{\mathcal{M}}=\oplus_{\mathbf{v}}\mathcal{F}_{\mathbf{v}}.
\]
As more orthogonal projectors are retained in the set, the resulting approximation by the maximal projection operator $\mathcal{M}$ will become better.

If the measure $\mu$ in (\ref{def-mu}) is taken to be the probability  measure
\[
d\mu(\Theta)=\rho(\Theta)d\Theta=\Pi_{i=1}^N \rho_i(\theta_i)d\theta_i,
\]
then the resulting HDMR defined in (\ref{HDMR-1}) is called ANOVA-HDMR \cite{rass99} .    Let us fix a point $\overline{\Theta}:=(\bar{\theta}_1, \bar{\theta}_2, \cdots, \bar{\theta}_N)$.
If the measure $\mu$ in (\ref{def-mu}) is taken as the Dirac measure located at the point $\Theta$, i.e.,
\[
d\mu(\Theta)=\Pi_{i=1}^N \delta(\theta_i-\bar{\theta}_i)d\theta_i,
\]
then the resulting HDMR is Cut-HDMR. The point $\overline{\Theta}$ is so called cut-point or anchor point.
In ANOVA-HDMR, the $\hmP_{\mathbf{v}} f$ involves $N-|\mathbf{v}|$ dimensional  integration and the computation of the components of $f$ is expensive.  In Cut-HDMR, the $\hmP_{\mathbf{v}} f$ only involves the evaluation in a $N-|\mathbf{v}|$ dimensional space and the computation is cheap and straightforward. Because of this reason,
we  use Cut-HDMR for the analysis and computation in the paper.

%\begin{rem}
%The (3) of Theorem \ref{HDMR-thm1} gives rise to the equality
%\begin{equation}
%\label{subvolume}
%\hmP_{\mfu} f=\sum_{\mathbf{v}\subseteq\mfu}\mathrm{M}_{\mathbf{v}} \hmP_{\mfu} f.
%\end{equation}
%By (\ref{subvolume}), it follows immediately that there is no error for  $|\mfu|$-th order Cut-HDMR approximation of $f(\Theta)$ whenever the point $\Theta$ is located on any
%$k$ ($k\leq |\mfu|$)-dimensional subvolume across the cut-point $\overline{\Theta}$ in $I^N$.
%\end{rem}

\subsection{A hybrid HDMR}

In this subsection, we use a less abstract representation of the HDMR to
motivate approximations suitable for practical computation. Specifically,
we review the traditional approach to truncated HDMR, and then propose and
analyze an alternative hybrid HDMR.

Expanding elements of (\ref{HDMR-1}), the HDMR of $f(\Theta)$ can be
written in the form,
 \begin{equation}
 \label{HDMR-2}
 f(\Theta)=f_0 +\sum_{i=1}^n f_i(\theta_i) +\sum_{i<j} f_{ij}(\theta_i, \theta_j) +\cdots  f_{12\cdots n} (\theta_1, \theta_2, \cdots, \theta_n).
 \end{equation}
Here $f_0$ is the zeroth-order component denoting the mean effect of  $f(\Theta)$. The first-order component $f_i(\theta_i)$ represents the individual
contribution of the input $\theta_i$ and the second-order component $f_{ij}(\theta_i, \theta_j)$ represents the cooperative effects of $\theta_i$ and $\theta_j$ and
so on.  We define $q$-th order truncated HDMR by
\begin{equation}
\label{truncated-HDMR}
\mathcal{M}_{q}f=f_0 +\sum_{m=1}^q \sum_{i_1<\cdots i_m} f_{i_1\cdots i_m} (\theta_{i_1} \cdots \theta_{i_m}).
\end{equation}
For  most realistic physical systems, low-order  HDMR $\mathcal{M}_{q}f$ (e.g., $q\leq 3$) may give a good approximation \cite{rass99}. The $\mathcal{M}_{q}f$ can approximate $f(\Theta)$ through truncating the HDMR.   The $\mathcal{M}_{q}f$ consists of a large number of component terms for high-dimensional models, the computation of all the components may be  costly.  For many cases, some high-order cooperative effects cannot be neglected in the model's output, especially when a model strongly relies on a few dependent variables.  Consequently, the traditional truncated HDMR in (\ref{truncated-HDMR}) may not yield a very good approximation.  To eliminate or alleviate these drawbacks of traditional truncated HDMR, we propose a new truncated HDMR, which
we refer to as hybrid HDMR.

We use the Fourier amplitude sensitivity test \cite{css75} to select
the most active dimensions from the components of the high-dimensional
random vector $\Theta$.  Let $\{f_j\}_{j=1}^N$ be the first-order
components defined in Eq.(\ref{HDMR-2}) and let $\sigma^2(f_j)$ denote
the variance of $f_j$.  We may assume that $\sigma^2(f_1)\geq
\sigma^2(f_2)\geq \cdots \geq \sigma^2(f_N)$ as we can re-order the
index set $\{j\}_{j=1}^N$ to satisfy this requirement.  Alternatively,
we can order the index set to produce monotonically decreasing
expectations, $E[f_1]\geq E[f_2]\geq \cdots \geq E[f_N]$, and use this
ordering to select the most active dimensions.  In this paper, we only
consider the variance sensitivity, and we calculate $\sigma^2(f_i)$
($i=1,\cdots N$) using the method described in
(\ref{comp-var}). Because the first-order components $\{f_i\}_{i=1}^N$
are defined in one-dimensional random parameter spaces, the
computational cost for $\{\sigma^2(f_i)\}_{i=1}^N$ is very
small. Moreover, most of the elements in $\{\sigma^2(f_i)\}_{i=1}^N$
will be reused when we calculate the variance of outputs represented
by hybrid HDMR.  We set a threshold constant with $0<\zeta <1$ and
find an optimal $J$ such that
\begin{equation}
\label{find-J}
\sum_{j=1}^J \sigma^2(\psi_j)/\sum_{j=1}^N \sigma^2(\psi_j) \geq \zeta.
\end{equation}
Then we define $\{\theta_j\}_{j=1}^J$ to be the most $J$ active
dimensions.  We note that $\sigma^2(f_j)$ provides information on the
impact of $\theta_j$ when it is acting alone on the output. It is
clear that if a change of the random input $\theta_j$ within its range
leads to a significant change in the output, then $\sigma^2(f_j)$ is
large. Therefore, (\ref{find-J}) provides a reasonable criteria for
identifying the most active dimensions, and for a given stochastic
model, we can adjust the value of $\zeta$ in Eq. (\ref{find-J}) such
that $J$ is much less than $N$. Defining the index set
$\mathcal{J}=\{1, 2, \cdots, J\}$, the proposed hybrid HDMR may be
written,
\begin{equation}
\label{new-HDMR1}
\mathcal{M}_{\mathcal{J}}f:=\sum_{\mathbf{v}\subseteq \mathcal{J}} \mathrm{M}_{\mathbf{v}}f+\sum_{i\in \{1,\cdots, N\}\setminus \mathcal{J}} \mathrm{M}_i f.
\end{equation}
Here the set of projectors $\{\mathrm{M}_{\mathbf{v}}, \mathbf{v}\subseteq \mathcal{J}\}\cup \bigg\{\mathrm{M}_i, i\in \{1,\cdots, N\}\setminus \mathcal{J}\bigg\}$ generate a lattice and its
 maximal projection operator is $\mathcal{M}_{\mathcal{J}}$. By (\ref{new-HDMR1}), the hybrid HDMR $\mathcal{M}_{\mathcal{J}}f$ consists of two parts: the first part  $\sum_{\mathbf{v}\subseteq \mathcal{J}} \mathrm{M}_{\mathbf{v}}f$ is the complete HDMR on the most active dimensions indexed in $\mathcal{J}$, the second part $\sum_{i\in \{1,\cdots, N\}\setminus \mathcal{J}} \mathrm{M}_i f$
 is the first-order truncated HDMR on the remaining dimensions $\{1,\cdots, N\}\setminus \mathcal{J}$.
 We note that the component $f_{\mathbf{u}}$ of $f$ defined in (\ref{HDMR-1})  can be recursively computed and  explicitly computed, respectively  by (ref. \cite{ksww09})
\begin{eqnarray}
\label{HDMR-express1}
f_{\mathbf{u}}=\hmP_{\mfu}f -\sum_{\mathbf{v}\subsetneq \mathbf{u}} f_{\mathbf{v}} \quad \text{and} \quad
f_{\mathbf{u}}=\sum_{\mathbf{v}\subseteq \mfu} (-1)^{|\mfu|-|\mathbf{v}|} \hmP_{\mathbf{v}} f.
\end{eqnarray}
 By (\ref{HDMR-express1}), we can rewrite equation (\ref{new-HDMR1}) by
 \begin{equation}
\label{new-HDMR2}
\mathcal{M}_{\mathcal{J}}f:=\hmP_{\mathcal{J}} f+\sum_{i\in \{1,\cdots, N\}\setminus \mathcal{J}} f_i(\theta_i).
\end{equation}
We note that the operator $\hmP_{\mathcal{J}}$ projects the $N$-variable function $f$ to a function defined on the $J$  the most active dimensions. The term $\hmP_{\mathcal{J}} f$  gives the
dominant contribution to $\mathcal{M}_{\mathcal{J}}f$.  Since any first-order components are usually important, these components are retained in $\mathcal{M}_{\mathcal{J}}f$.
In Cut-HDMR, there is no error for the approximation $\mathcal{M}_{\mathcal{J}}f$  of $f(\Theta)$ whenever the point $\Theta$ is located the $J$-th dimensional subvolume
$\{ \Theta_{\mathcal{J}}, \bar{\theta}_{J+1}, \cdots, \bar{\theta}_N\}$ across the cut-point $\overline{\Theta}$.

In this paper, identification of the $J$ most important dimensions is based on a global sensitivity analysis on the univariate terms of HDMR.  This criteria has been
shown to be reasonable for many cases and is often used \cite{css75,  mz10,yclk11,gh10}.
However, it is important to note that this criteria may not effectively identify the most important dimensions in some cases.  For example, in some situations the individual contribution of an input parameter may not be significant, even though its cooperative influence with other inputs is significant.  Similarly, the identification of the most active dimensions for highly variable solutions may require different criteria. Recent work \cite{wl13} presents a preliminary discussion on the choice of this critera. Since, the new hybrid HDMR naturally includes the cooperative effects of higher-order terms within the $J$ most important dimensions, this selection criteria is key to addressing these challenging cases and is an important topic for furture work.

If we use traditional truncated HDMR to approximate the term $\hmP_{\mathcal{J}} f$ in (\ref{new-HDMR2}), then we can obtain the adaptive  HDMR developed in \cite{mz10, yclk11},
\begin{equation}
\label{ad-HDMR}
\mathcal{M}^{ad}_{\mathcal{J},q}f:=\sum_{\mathbf{v}\subseteq \mathcal{J},  |\mathbf{v}|\leq q} \mathrm{M}_{\mathbf{v}}(\hmP_{\mathcal{J}} f)+\sum_{i\in \{1,\cdots, N\}\setminus \mathcal{J}} f_i(\theta_i).
\end{equation}
To simplify the notation, we will suppress $q$ in   $\mathcal{M}^{ad}_{\mathcal{J},q}$ in the paper when the truncation order $q$ is not emphasized.  The following proposition gives the relationship among $f$, $\mathcal{M}_{\mathcal{J}}f$ and $\mathcal{M}^{ad}_{\mathcal{J}}f$.
\begin{proposition}
Let $\mathcal{M}_{\mathcal{J}}f$ and $\mathcal{M}^{ad}_{\mathcal{J}}f$ be defined in (\ref{new-HDMR2}) and (\ref{ad-HDMR}), respectively. Then
\begin{equation}
\label{eq0-ad}
\|f-\mathcal{M}^{ad}_{\mathcal{J}}f\|_{\mathcal{F}}^2=\|f-\mathcal{M}_{\mathcal{J}}f\|_{\mathcal{F}}^2+\|\mathcal{M}_{\mathcal{J}}f-\mathcal{M}^{ad}_{\mathcal{J}}f\|_{\mathcal{F}}^2.
\end{equation}
\end{proposition}
\begin{proof}
We note the equality
\begin{equation}
\label{eq1-ad}
f-\mathcal{M}^{ad}_{\mathcal{J}}f=(f-\mathcal{M}_{\mathcal{J}}f) +(\mathcal{M}_{\mathcal{J}}f-\mathcal{M}^{ad}_{\mathcal{J}}f).
\end{equation}
Because the set of HDMR components in $f-\mathcal{M}_{\mathcal{J}}f$ do not intersect with the set of HDMR components in $\mathcal{M}_{\mathcal{J}}f-\mathcal{M}^{ad}_{\mathcal{J}}f$,
The equation  (\ref{Decomp-F}) implies that
\begin{equation}
\label{eq2-ad}
 \langle f-\mathcal{M}_{\mathcal{J}}f,   \mathcal{M}_{\mathcal{J}}f-\mathcal{M}^{ad}_{\mathcal{J}}f \rangle =0.
\end{equation}
Then the equation (\ref{eq0-ad}) follows immediately by combining equation (\ref{eq1-ad}) and (\ref{eq2-ad}).
\end{proof}

The equation (\ref{eq0-ad}) implies  that hybrid HDMR has better approximation properties than adaptive  HDMR.

The mean of $f$ can be computed by summing the mean of all HDMR components of $f$ for both ANOVA-HDMR and Cut-HDMR.  The variance of $f$ is the sum of the variance of
HDMR components of $f$ for ANOVA-HDMR.  However, the direct summation of variance of components of $f$ may not equal to variance of $f$  for Cut-HDMR. Consequently,  we may want to have a truncated ANOVA-HDMR to approximate the variance of $f$  for Cut-HDMR.   If we directly derive a truncated ANOVA-HDMR, the computation is very costly and more expensive than the direct computation of the variance of $f$ itself. This is because a truncated ANOVA-HDMR requires computing  many high-dimensional integrations.  To overcome the difficulty,  we can use  an efficient  two-step approach to  derive a hybrid ANOVA-HDMR through the hybrid Cut-HDMR.

Let $\Theta_{\mathcal{J}}$ be the $J$-dim variable with indices in $\mathcal{J}$.  Using the hybrid HDMR formulation (\ref{new-HDMR2}),  the hybrid  Cut-HDMR $f^{cut}(\theta)$ of $f$ has the following form
\begin{equation}
\label{cut}
f^{cut}(\Theta)=\hat{f}^{cut}(\Theta_{\mathcal{J}})+\sum_{i\in \{1,\cdots, N\}\setminus \mathcal{J}} f_i^{cut}(\theta_i),   \quad \text{where $\hat{f}^{cut}(\Theta_{\mathcal{J}}):=\hmP_{\mathcal{J}} f$}.
\end{equation}
We use $\mathcal{M}_{\mathcal{J}}$ to act on $f^{cut}(\Theta)$ to get a hybrid ANOVA-HDMR $f^{anova}(\theta)$ of $f$, which has the following form
\begin{equation}
\label{anova}
f^{anova}(\Theta):=\mathcal{M}_{\mathcal{J}}f^{cut}=\hat{f}^{anova}(\Theta_{\mathcal{J}})+\sum_{i\in \{1,\cdots, N\}\setminus \mathcal{J}} f_i^{anova}(\theta_i).
\end{equation}
Then we have the following theorem.
\begin{theorem}
\label{HDMR-thm3}
Let $f^{cut}(\Theta)$ and $f^{anova}(\Theta)$ be defined in (\ref{cut}) and (\ref{anova}), respectively.  Then
\begin{eqnarray}
E[f^{cut}]=E[f^{anova}]=E[\hat{f}^{cut}]+\sum_{i\in \{1,\cdots, N\}\setminus \mathcal{J}} E[f_i^{cut}], \label{mean-eq}\\
\mbox{Var}[f^{cut}]=\mbox{Var}[f^{anova}]=\mbox{Var}[\hat{f}^{cut}]+\sum_{i\in \{1,\cdots, N\}\setminus \mathcal{J}} \mbox{Var}[f_i^{cut}]. \label{var-eq}
\end{eqnarray}
\end{theorem}

The proof of Theorem \ref{HDMR-thm3} is presented in Appendix \ref{app1}.

Theorem \ref{HDMR-thm3} implies that we can compute the mean and
variance of $f^{cut}$ by directly summing the means and variances of
the components of $f^{cut}$.  This observation can significantly reduce the
complexity of this computation.  In addition, the derived hybrid ANOVA-HDMR
$f^{anova}$ is easily obtained by using the two-step approach through
the hybrid Cut-HDMR $f^{cut}$.  Our further calculation shows that
the traditional truncated Cut-HDMR (\ref{truncated-HDMR}) and adaptive
Cut-HDMR (\ref{ad-HDMR}) do not have these properties.  We use sparse-grid
quadrature \cite{gg98} to compute the mean and variance in the paper.

Compared with the traditional truncated HDMR and adaptive HDMR, the
hybrid HDMR defined in Eq.(\ref{new-HDMR2}) has the following
advantages: (a) The hybrid HDMR consists of significantly fewer terms
than the traditional truncated HDMR and adaptive HDMR, the total
computational effort of the hybrid HDMR can be substantially reduced.
(b) Since the hybrid HDMR inherently includes all the cooperative
contributions from the most active dimensions, approximation accuracy
is not worse (maybe better) in hybrid HDMR than the traditional
truncated HDMR and adaptive HDMR; (c) According to Theorem
\ref{HDMR-thm3}, the computation of variance of hybrid HDMR is much
more efficient.

\subsection{Analysis of computational complexity}

In the previous subsection, we have addressed the accuracy of the hybrid HDMR and developed some comparisons with traditional truncated HDMR and adaptive HDMR.
In this subsection, we discuss the computational efficiency for the various HDMR techniques when Smolyak sparse-grid collocation is used.  We remind the reader that the HDMR refers to Cut-HDMR.

Let $C(f, \ell)$ be the number of sparse grid collocation points with level $\ell$ in full random parameter dimension space and $C(\mathcal{M}_q f, \ell)$ the total number of sparse
grid collocation points with level $\ell$ in the traditional truncated HDMR $\mathcal{M}_q f$.  We define $C(\mathcal{M}_{\mathcal{J}} f, \ell)$ and $C(\mathcal{M}^{ad}_{\mathcal{J},q} f, \ell)$
in a similar way.  We first consider the case $\ell=2$ and $q=2$  and calculate the total number of sparse grid collocation points for the different approaches.  By using (\ref{coll-nodes}), we have for $N\gg 1$
\begin{eqnarray*}
C(f, 2)&=&H(N+2,N)\approx 2 N^2,\\
C(\mathcal{M}_2 f, 2)&=&\sum_{j=1}^2 {N\choose j}H(j+2,j)=5N+13{N\choose 2}\approx 13N^2,\\
C(\mathcal{M}_{\mathcal{J}} f, 2)&=&H(J+2,J)+(N-J)H(1+2,1)\approx 2J^2 +5(N-J),\\
C(\mathcal{M}^{ad}_{\mathcal{J},2} f, 2)&=&\sum_{j=1}^2 {J\choose j}H(j+2,j)+(N-J)H(1+2,1)\approx 13J^2+5N.
\end{eqnarray*}
Consequently,  if $J\approx N/2$, it follows immediately that
\[
C(\mathcal{M}_{\mathcal{J}} f, 2)< C(f, 2)<C(\mathcal{M}^{ad}_{\mathcal{J},2} f, 2)<C(\mathcal{M}_2 f, 2).
\]
This means that the hybrid HDMR $\mathcal{M}_{\mathcal{J}} f$ requires
the smallest number of collocation points when $\ell=2$ and $q=2$.
This case is of particular interest in this paper.

Next we consider two other cases of interest.  First, it can be shown that if $J\approx 30$ and $N>(\frac{J}{3})^{1.5}$, then
\[
C(\mathcal{M}_{\mathcal{J}} f, 3)<C(\mathcal{M}^{ad}_{\mathcal{J},2} f, 3)<C(\mathcal{M}_2 f, 3)<C(f, 3).
\]
This implies that the computational effort in the hybrid HDMR
$\mathcal{M}_{\mathcal{J}} f$ is the least for $\ell=3$ when the
number of the most active dimensions is moderate for a
high-dimensional problem. However, if the number $J$ of the most
active dimensions is large (e.g., $J\gg30$), we can show that
\begin{equation}
\label{large-J}
C(\mathcal{M}^{ad}_{\mathcal{J},q} f, \ell)<C(\mathcal{M}_q f, \ell)<C(\mathcal{M}_{\mathcal{J}} f, \ell)<C(f, \ell), \quad \text{where $\ell\geq 3$ and $q\leq 3$}.
\end{equation}
This relationship tells us that adaptive HDMR
$\mathcal{M}^{ad}_{\mathcal{J},q}f$ may be the most efficient if
higher-level collocation is required and $J$ is large as well.

For a high-dimensional stochastic model, if the number $J$ of the most
active dimensions is large and high-level (e.g., level 3 and above)
sparse-grid collocation is required to approximate the term
$\hmP_{\mathcal{J}} f$ in (\ref{new-HDMR2}), using the adaptive HDMR
(\ref{ad-HDMR}) may improve efficiency according to the bounds given in
(\ref{large-J}).  Adaptive HDMR has been extensively considered in
many recent papers \cite{gh10,mz10,mz11, yclk11}.  However, in our experience,
the number of the most active dimensions $J$ is often
less than ${N\over 2}$ as $N$ is large.  If $\hmP_{\mathcal{J}} f$ is
smooth with respect to $\Theta_{\mathcal{J}}$, low-level sparse grid
collocation (e.g., level 2) will generally provide an accurate
approximation.  Many problems fall into this class, and we focus
on them in this paper.

\subsection{Integrating HDMR and sparse-grid collocation}
\label{sect-HDMR-sgc}

As stated in Subsection \ref{sect-sgc}, the sparse grid stochastic
collocation method can reduce to the stochastic two-phase flow system
(\ref{tp-system}) into a set of deterministic two-phase flow systems.
However, it suffers from {\it curse of dimensionality} with increasing
stochastic dimension.  Integrating HDMR and a sparse-grid collocation
methods provides an approach  to overcome this difficulty and may
significantly enhance the efficiency.

Without loss of generality, we consider the saturation solution $S(\Theta)$ as an example to present the technique for the hybrid Cut-HDMR.
Using (\ref{new-HDMR2}), (\ref{HDMR-express1}) and Smolyak sparse-grid interpolation,  we have
 \begin{eqnarray}
 \label{approx-eq1}
 \begin{split}
  S(\Theta)&\approx \mathcal{M}_{\mathcal{J}} S(\Theta)=\hat{S}(\Theta_{\mathcal{J}})+\sum_{i\in \{1,\dots, N\}\setminus \mathcal{J}} S_i(\theta_i)\\
  &=\hat{S}(\Theta_{\mathcal{J}})+ \sum_{i\in \{1,\dots, N\}\setminus \mathcal{J}} \hat{S}(\theta_i)-(N-J)S_0\\
  &\approx \mathcal{A}(J+\ell, J)\big[\hat{S}(\Theta_{\mathcal{J}})\big]+\sum_{i\in \{1,\dots, N\}\setminus \mathcal{J}}\mathcal{A}(1+\ell, 1)\big[ \hat{S}(\theta_i)\big]-(N-J)S_0,
  \end{split}
 \end{eqnarray}
where $\hat{S}(\Theta_{\mathcal{J}})={\hmP_{\mathcal{J}}S(\Theta)}$, $\hat{S}(\theta_i)=\hmP_{i}S(\Theta)$  and $S_0=\hmP_{\emptyset}S(\Theta)=S(\overline{\Theta})$. Here $\overline{\Theta}$
is the cut-point for a Cut-HDMR. The choice of the cut-point may
affect the accuracy of the truncated Cut-HDMR approximation. The study
in \cite{gao10} argued that an optimal choice of the cut-point is the
center point of a sparse-grid quadrature.  In this paper, we will use
such a cut-point for computation.  Due to Theorem \ref{HDMR-thm3}, the
mean and variance of $S$ can be approximated by
\begin{eqnarray}
E[S(\Theta)] &\approx & E\bigg[\mathcal{A}(J+\ell, J)\big[\hat{S}(\Theta_{\mathcal{J}})\big]\bigg]+\sum_{i\in \{1,\dots, N\}\setminus \mathcal{J}} E\bigg[ \mathcal{A}(1+\ell, 1)\big[ \hat{S}(\theta_i)\big]\bigg] \nonumber\\
&-&(N-J)S_0, \label{mean-comp1}\\
\mbox{Var}[S(\Theta)] &\approx& \mbox{Var}\bigg[\mathcal{A}(J+\ell, J)\big[\hat{S}(\Theta_{\mathcal{J}})\big]\bigg]+\sum_{i\in \{1,\dots, N\}\setminus \mathcal{J}} \mbox{Var}\bigg[ \mathcal{A}(1+\ell, 1)\big[ \hat{S}(\theta_i)\big]\bigg] \nonumber.
\end{eqnarray}
From (\ref{approx-eq1}), we find that the approximation error of the mean and variance comes from two sources: truncated HDMR and Smolyak sparse grid quadrature.
Let $\{\Theta_{\mathcal{J}}^{(j)}, w^{(j)}\}_{j=1}^{N_J}$ be the set of pairs of collocation points and weights in the most active subspace $I^J$ and $\{\theta_i^{(k)}, w^{(k)}\}_{k=1}^{N_i}$  the set of pairs of collocation
points and weights in $I$.   We note that $N_J=H(J+\ell, J)$ and $N_i=H(1+\ell, 1)$. Then by (\ref{mean-comp1}),  the mean of $S$ can be computed by
\begin{equation}
\label{mean-comp2}
E[S(\Theta)] \approx \sum_{j=1}^{N_J} \hat{S}(\Theta_{\mathcal{J}}^{(j)})w^{(j)}+\sum_{i\in \{1,\dots, N\}\setminus \mathcal{J}}\sum_{k=1}^{N_i}\hat{S}(\theta_i^{(k)})w^{(k)}-(N-J)S_0.
\end{equation}
Similarly, we can  compute $\mbox{Var}[S(\Theta)]$ in terms of evaluations of $\hat{S}(\Theta_{\mathcal{J}})$ and $\hat{S}(\theta_i)$ at the collocation points.

Because each of the terms in adaptive HDMR
$\mathcal{M}^{ad}_{\mathcal{J},q}S$ are usually correlated,
the variance of $\mathcal{M}^{ad}_{\mathcal{J},q}S$ is not
equal to the sum of the variance of each term in (\ref{ad-HDMR}). Let
$\{\Theta_i^{(n)}, w^{(n)}\}_{n=1}^{N_c}$ (where $N_c=H(N+\ell, N)$)
be the set of pairs of collocation points and weights in the full
random dimension $I^N$. To compute the variance using adaptive HDMR, we need
to project each collocation point $\Theta_i^{(n)}$ onto the components
$\{\Theta_{\mathbf{v}}^{(n)}: \mathbf{v}\subseteq \mathcal{J},
|\mathbf{v}|\leq q\}$ and $\{\theta_i^{(n)}: i\in \{1,\cdots,
N\}\setminus \mathcal{J}\}$ and interpolate all terms in
(\ref{ad-HDMR}).  Then we use (\ref{comp-var}) to calculate the
variance.  This process involves at most $H(N+\ell, N)\big[
N-J+\sum_{j=1}^q{J \choose j}\big]$ Smolyak sparse-grid
interpolations.  The computation of these interpolations is usually
very expensive when $N$ and $J$ are large, and hence, computing the
variance using adaptive HDMR is usually much more expensive than using
hybrid HDMR.  The numerical experiments in Section \ref{sect-num}
demonstrate this performance advantage for hybrid HDMR.

%%%%%%%%%%%%%%
%%%%%%%%%%%%%

\section{Mixed multiscale finite element method}

In Section \ref{sect-HDMR-sgc}, we have discussed integrating hybrid  HDMR and sparse grid collocation method to reduce the computation complexity from high-dimensional stochastic spaces.
The permeability field is often heterogeneous in porous media. It is necessary to use a numerical method to capture the heterogeneity.  MMsFEM is one of such numerical methods and has been widely used in simulating flows in heterogeneous porous media \cite{aej08, jml}.  To simulate the two-phase flow system (\ref{tp-system}), it is necessary to retain local conservation for
velocity (or flux).  To this end, we use MMsFEM to solve the flow equation and obtain locally conservative velocity. Using MMsFEM coarsens the multiscale model in spatial space and can significantly enhance the computation efficiency.

Corresponding to the hybrid HDMR expansion of $S(\Theta)$ in (\ref{approx-eq1}), i.e.,
\[
\mathcal{M}_{\mathcal{J}} S(\Theta)=\hat{S}(\Theta_{\mathcal{J}})+\sum_{i\in \{1,\dots, N\}\setminus \mathcal{J}} \hat{S}(\theta_i)-(N-J)S_0,
\]
the velocity $u(\Theta)$ in (\ref{tp-system}) also admits the same hybrid  HDMR expansion as  following
\begin{equation}
\label{HDMR-u}
\mathcal{M}_{\mathcal{J}} u(\Theta)=\hat{u}(\Theta_{\mathcal{J}})+ \sum_{i\in \{1,\dots, N\}\setminus \mathcal{J}} \hat{u}(\theta_i)-(N-J)u_0,
\end{equation}
where $\hat{u}(\Theta_{\mathcal{J}})={\hmP_{\mathcal{J}}u(\Theta)}$, $\hat{u}(\theta_i)=\hmP_{i}u(\Theta)$  and $u_0=u(\overline{\Theta})$. We use $\hat{u}(\Theta_{\mathcal{J}})$ to
obtain $\hat{S}(\Theta_{\mathcal{J}})$,  $\hat{u}(\theta_i)$ to obtain $\hat{S}(\theta_i)$ and $u_0$ to obtain $S_0$.

 Without loss of generality, we may assume that the boundary condition in the flow equation of (\ref{approx-eq1}) is no flow boundary condition.  Let $\hat{k}(\Theta_{\mathcal{J}})={\hmP_{\mathcal{J}}k(\Theta)}$ and $\hat{k}(\theta_i)=\hmP_{i}k(\Theta)$.   Then we can uniformly formulate
  the mixed formulation of equations of  $\hat{u}(\Theta_{\mathcal{J}})$ and $\hat{u}(\theta_i)$   as following,
 \begin{eqnarray}
 \label{mixed-eq1}
\begin{split}
\begin{cases}
& \big(\lambda(\hat{S})\hat{k}\big)^{-1} \hat{u}+\nabla \hat{p} = 0 \ \ \text{in} \ D\\
& \mbox{div}(\hat{u}) = q  \ \ \text{in} \ D \\
& \hat{u} \cdot n = 0 \ \ \text{on} \ \partial D.
 \end{cases}
 \end{split}
\end{eqnarray}
Here $\hat{u}(\Theta_{\mathcal{J}})$ and $\hat{u}(\theta_i)$ are corresponding to the coefficients $\hat{k}(\Theta_{\mathcal{J}})$ and $\hat{k}(\theta_i)$, respectively.
Let $k_0=k(\overline{\Theta})$. Then $u_0$ is the solution to  equation (\ref{mixed-eq1}) if $\hat{k}$ and $\hat{S}$ are replaced by $k_0$ and $S_0$, respectively.

The weak mixed formulation of (\ref{mixed-eq1}) reads:  seek $(\hat{u}, \hat{p})\in H_0(div, D) \times L^2(D)/R$ such that they satisfy the equation
 \begin{eqnarray*}
 \begin{split}
\begin{cases}
& \bigg( \big(\lambda(\hat{S})\hat{k}\big)^{-1}\hat{u},  v\bigg)- \big(\mbox{div}(v), \hat{p}\big) =0   \quad  \forall  v\in H_0(div, D) \\
& \big(\mbox{div} (\hat{u}), r) = (q,   r)\ \ \ \forall r\in L^2(D).
\end{cases}
 \end{split}
\end{eqnarray*}
Let $V_h\subset H_0(\mbox{div}, D)$ and $Q_h\subset L^2(D)/R$ be the finite element  spaces for velocity and pressure, respectively.
Then the numerical mixed formulation of (\ref{mixed-eq1}) is  to find $(\hat{u}_h, \hat{p}_h)\in V_h\times Q_h$ such that they satisfy
\begin{eqnarray}
 \label{num-weak}
 \begin{split}
\begin{cases}
& \bigg( \big(\lambda(\hat{S})\hat{k}\big)^{-1}\hat{u}_h,  v_h\bigg)- \big(\mbox{div}(v_h), \hat{p}_h\big) =0   \quad  \forall  v_h\in V_h \\
& \big(\mbox{div} (\hat{u}_h), r_h) = (q,   r_h)\ \ \ \forall r_h\in Q_h.
\end{cases}
 \end{split}
\end{eqnarray}

We use MMsFEM for (\ref{num-weak}). It means that mixed finite element approximation
is performed on coarse grid, where the multiscale basis functions are defined.
In MMsFEM, we use piecewise constant basis functions on coarse grid for pressure.
For the velocity, we define multiscale velocity basis functions.  The degree of freedom
of multiscale velocity basis function is defined on interface of coarse grid.
Let $e_i^K$ be a generic  edge or  face of the coarse  block $K$. The velocity multiscale basis equation
associated with $e_i^K$ is defined by
\begin{eqnarray}
\label{basis-eq}
\begin{split}
\begin{cases}
&-\mbox{div}(\hat{k} \nabla w_{i}^{K}) =  \frac{1}{|K|} \quad  \text{in}\ K  \\
&-\hat{k}\nabla w_{i}^{K}\cdot n  =  \left\{
\begin{array}{ll}
\frac{b_i^K\cdot n} {\int_{e_i^K} b_i^K \cdot n ds} \ \ \text{on}\ e_i^K \\
0 \ \ \text{else}.
\end{array}%
\right.
\end{cases}
\end{split}
\end{eqnarray}
For local mixed MsFEM \cite{ch03},  $b_i^K=n$, the normal vector. If the media demonstrate strong non-local features including
channels, fracture and shale barriers,  some limited  global information is needed to define the boundary condition $b_i^K$ to improve
 accuracy of approximation \cite{aej08, jem10}. We will specify the boundary condition $b_i^K$ for different parts in (\ref{HDMR-u}).
Then $\psi_i^K=-\hat{k}\nabla w_i^K$  defines  multiscale velocity basis function  associated to $e_i^K$, and the multiscale finite dimensional
space for velocity is defined by
\[
V_h=\bigoplus_{K, i} \psi_i^K.
\]

It is well-known that using approximate single-phase global velocity information can considerably improve accuracy for multiscale simulation of two-phase flows \cite{aa, aej08, jem10,jml}.
For the two-phase flow system (\ref{tp-system}),  the single-phase global velocity $u_{sg}(\Theta)$ solves  the following equation
\begin{eqnarray}
 \label{global-sg}
\begin{split}
\begin{cases}
& \big(k(\Theta)\big)^{-1} u_{sg}+\nabla p_{sg} = 0 \ \ \text{in} \ D\\
& \mbox{div}(u_{sg}) = q  \ \ \text{in} \ D \\
& u_{sg} \cdot n = 0 \ \ \text{on} \ \partial D.
 \end{cases}
 \end{split}
\end{eqnarray}
By using hybrid  HDMR and sparse grid interpolation,  the $u_{sg}(\Theta)$ admits the following approximation
\begin{eqnarray*}
\mathcal{M}_{\mathcal{J}} u_{sg}(\Theta)&=&\hat{u}_{sg}(\Theta_{\mathcal{J}})+ \sum_{i\in \{1,\dots, N\}\setminus \mathcal{J}} \hat{u}_{sg}(\theta_i)-(N-J)u_{sg,0}\\
&\approx & \bigg \{u_{sg,0}+\sum_{j\in\mathcal{J}} u_{sg, j}(\theta_j)\bigg\} + \sum_{i\in \{1,\dots, N\}\setminus \mathcal{J}} \hat{u}_{sg}(\theta_i)-(N-J)u_{sg,0}\\
&=& \bigg\{ (1-J)u_{sg,0}+\sum_{j\in\mathcal{J}} \hat{u}_{sg}(\theta_j)\bigg\}+\sum_{i\in \{1,\dots, N\}\setminus \mathcal{J}} \hat{u}_{sg}(\theta_i)-(N-J)u_{sg,0}\\
&\approx & \bigg\{ (1-J)u_{sg,0}+\sum_{j\in\mathcal{J}} \mathcal{A}(1+\ell, 1)\big[\hat{u}_{sg}(\theta_j)\big] \bigg\} \\
&+&\sum_{i\in \{1,\dots, N\}\setminus \mathcal{J}} \mathcal{A}(1+\ell', 1)\big[\hat{u}_{sg}(\theta_i)\big]-(N-J)u_{sg,0}
\end{eqnarray*}
where $\hat{u}_{sg}(\theta_j)=\hmP_{j}u_{sg}(\Theta)$ and $\hat{u}_{sg}(\theta_i)=\hmP_{i}u_{sg}(\Theta)$.  Here the interpolation levels
$\ell$ and $\ell'$ can be different. Because $\{\theta_j\}_{j\in \mathcal{J}}$ are the most active dimensions,  it is usually desirable that $\ell\geq \ell'$.
Now we are ready to specifically describe the multiscale finite element space for $\hat{u}(\Theta_{\mathcal{J}})$,  $\hat{u}(\theta_i)$ ($i\in \{1,\dots, N\}\setminus \mathcal{J}$) and $u_0$.
\begin{itemize}
\item To construct the multiscale basis functions for $\hat{u}(\Theta_{\mathcal{J}})$,  we take $\hat{k}=\hat{k}(\Theta_{\mathcal{J}})$ in (\ref{basis-eq}) and
\begin{equation}
\label{bc1}
b_i^K=b_i^K(\Theta_{\mathcal{J}})=\bigg( (1-J)u_{sg,0}+\sum_{j\in\mathcal{J}} \mathcal{A}(1+\ell, 1)\big[\hat{u}_{sg}(\theta_j)\big] \bigg)|_{e_i^K}.
\end{equation}
The multiscale finite element space for $\hat{u}(\Theta_{\mathcal{J}})$ is defined by
\begin{equation}
\label{MsFS1}
V_h(\Theta_{\mathcal{J}})=\bigoplus_{K, i} \psi_i^K(\Theta_{\mathcal{J}}), \quad \text{where $\psi_i^K(\Theta_{\mathcal{J}})=-\hat{k}(\Theta_{\mathcal{J}})\nabla w_i^K(\Theta_{\mathcal{J}})$}.
\end{equation}
\item To construct the multiscale basis functions for $\hat{u}(\theta_i)$  ($i\in \{1,\dots, N\}\setminus \mathcal{J}$),  we take $\hat{k}=\hat{k}(\theta_i)$ in (\ref{basis-eq}) and
\begin{equation}
\label{bc2}
b_i^K=b_i^K(\theta_i)=\bigg(\mathcal{A}(1+\ell', 1)\big[\hat{u}_{sg}(\theta_i)\big]\bigg)|_{e_i^K}.
\end{equation}
The multiscale finite element space for $\hat{u}(\theta_i)$ is defined by
\begin{equation}
\label{MsFS2}
V_h(\theta_i)=\bigoplus_{K, i} \psi_i^K(\theta_i), \quad \text{where $\psi_i^K(\theta_i)=-\hat{k}(\theta_i)\nabla w_i^K(\theta_i)$}.
\end{equation}
\item To construct the multiscale basis functions for $u_0$,  we replace  $\hat{k}$ by $k(\overline{\Theta})$ in (\ref{basis-eq}) and
\[
b_i^K=u_{sg,0}|_{e_i^K}.
\]
The multiscale finite element space for $u_0$ is defined by
\[
V_h(\overline{\Theta})=\bigoplus_{K, i} \psi_i^K(\overline{\Theta}), \quad \text{where $\psi_i^K(\overline{\Theta})=-k(\overline{\Theta})\nabla w_i^K(\overline{\Theta})$}.
\]
\end{itemize}

For different parts of the hybrid  HDMR expansion $\mathcal{M}_{\mathcal{J}} u(\Theta)$ defined in (\ref{HDMR-u}), we use different boundary conditions for  multiscale basis equations. This
increases the hierarchies of multiscale basis functions, which are in tune with the sensitivity of random parameter dimensions.

By (\ref{mean-comp2}), to compute the  moments  of outputs we only need to construct the multiscale finite element space $V_h(\Theta_{\mathcal{J}})$ at the set of collocation point $\{\Theta_{\mathcal{J}}^{(j)}\}_{j=1}^{N_J}$, where $N_J=H(J+\ell, J)$.  For arbitrary  $\Theta_{\mathcal{J}}\in I^J$, the boundary condition $b_i^K(\Theta_{\mathcal{J}})$
is completely determined by $u_{sg}(\theta_j^{(m)})$, where $j\in \mathcal{J}$ and $m=1,\dots, H(1+\ell, 1)$. Here $\{\theta_j^{(m)}\}$ is a set of collocation points in $I$.
Similarly,    we also need to construct the multiscale finite element space $V_h(\theta_i)$ at the set of 1-dimensional  collocation points $\{ \theta_i^{(j')}\}$, where $i\in \{1,\dots, N\}\setminus \mathcal{J}$ and $j'=1,\dots,  H(1+\ell, 1)$.   For arbitrary  $\theta_i\in I$, the boundary condition $b_i^K(\theta_i)$
is completely determined by $u_{sg}(\theta_i^{(m')})$, where $m'=1,\dots, H(1+\ell', 1)$ and $\{\theta_j^{(m')}\}$ is a set of collocation points in $I$.
These single-phase global velocity information $\{u_{sg}(\theta_j^{(m)}), u_{sg}(\theta_i^{(m')})\}$ is pre-computed  and can be repeatedly used in
the whole stochastic flow simulations.

The convergence analysis for the mixed MsFEM using approximate global information can be found in \cite{jem10,jml}.
From the above discussion, here the approximate global information is completely determined by stochastic single-phase velocity contributed from individual random dimensions $\theta_i$ ($i=1, \dots, N$). According to the analysis in \cite{jem10}, the resonance error between the coarse mesh size $h$ and the effect of individual contributions of $\theta_i$ ($i=1, \dots, N$)
 is negligible.  This error is a major source using local MMsFEM \cite{jml}.  However, the resonance error between the coarse mesh size $h$ and the effect of cooperative  contributions of $\Theta$ is retained, but is usually  small.  This is our motivation to use limited global information determined  in 1-dimensional random space.
If the random permeability field  has low variance and does not exhibit  global features,  we can use local mixed MsFEM for
the stochastic simulation.

Define
\[
\widetilde{u}_h(x, \Theta):=\mathcal{A}(J+\ell, J)\big[\hat{u}_h(x, \Theta_{\mathcal{J}})\big]+\sum_{i\in \{1,\dots, N\}\setminus \mathcal{J}}\mathcal{A}(1+\ell, 1)\big[ \hat{u}_h(x,\theta_i)\big]-(N-J)u_{0,h}(x).
\]
For velocity, we actually compute $\widetilde{u}_h(x, \Theta)$, namely,  the numerical solutions  $\hat{u}_h(x, \Theta_{\mathcal{J}})$ and $\hat{u}_h(x,\theta_i)$ are evaluated  at their collocation points.
We define $\widetilde{u}(x, \Theta)$ in the same way.
Now we analyze the total error between $u(x, \Theta)$ and $\widetilde{u}_h(x, \Theta)$. By triangle inequality,
\begin{eqnarray}
\begin{split}
&E\big[ \|u(x, \Theta)-\widetilde{u}_h(x, \Theta)\|_{L^2(D)}\big]\\
&\leq E\big[ \|u(x, \Theta)-\mathcal{M}_{\mathcal{J}}u(x, \Theta)\|_{L^2(D)}\big]
+E\big[ \|\mathcal{M}_{\mathcal{J}}u(x, \Theta)-\widetilde{u}(x, \Theta)\|_{L^2(D)}\big]\\
&+E\big[ \|\widetilde{u}(x, \Theta)-\widetilde{u}_h(x, \Theta)\|_{L^2(D)}\big]\\
&:= \mathcal{E}_{hdmr}+\mathcal{E}_{col}+\mathcal{E}_{ms},
\end{split}
\end{eqnarray}
where $\mathcal{E}_{hdmr}$ represents the error from hybrid HDMR, $\mathcal{E}_{col}$ the error by sparse grid collocation and $\mathcal{E}_{ms}$
the error of MMsFEM.  The error $\mathcal{E}_{hdmr}$ depends on the choice of the most active dimensions $\mathcal{J}$,  the error $\mathcal{E}_{col}$ depends
on the number of collocation points and random dimensions,  and the error  $\mathcal{E}_{ms}$ relies on the coarse mesh size $h$ and the approximation of the boundary conditions (\ref{bc1}) and (\ref{bc2}) to $\hat{u}_{sg}(\Theta_{\mathcal{J}})$ and
$\hat{u}_{sg}(\theta_i)$, respectively.  According to the works in \cite{aej08, jem10}, if the boundary conditions in (\ref{bc1}) and (\ref{bc2}) are replaced by $\hat{u}_{sg}(\Theta_{\mathcal{J}})$ and
$\hat{u}_{sg}(\theta_i)$, respectively,  then  the  error $\mathcal{E}_{ms}$ only depends on coarse mesh size $h$.
The error $\mathcal{E}_{col}$  has been extensively discussed in many recent literatures, e.g., \cite{bnt07, bnr00, xiu05}, and we focus on the errors $\mathcal{E}_{hdmr}$ and $\mathcal{E}_{ms}$
in the paper.

\section{Numerical results}
\label{sect-num}

In this section, we present numerical results for two-phase flow in
random porous media.  The hybrid HDMR and MMsFEM are incorporated to
enhance simulation efficiency.  For random porous media with non-local
spatial features and high variance, we use limited global information
to capture the multiscale velocity information and achieve better
accuracy.  We combine MMsFEM with the two different truncated HDMR
techniques techniques (hybrid HDMR $\mathcal{M}_{\mathcal{J}}$ and
adaptive HDMR $\mathcal{M}^{ad}_{\mathcal{J}}$) for the flow
simulations. The accuracy and efficiency of the different methods will
be compared.

For numerical simulation,  we assume that $k(x, \Theta)$ is a logarithmic random  field, i.e., $k(x, \Theta):=\exp(a(x, \Theta))$. This
assumption assures the positivity of  $k(x, \Theta)$ and the
well-posedness of the flow equation in (\ref{tp-system}).
Here  $a(x, \omega)$ is a stochastic field  and can be characterized by its covariance function
$\mbox{cov}[a] $: $\bar{D}\times \bar{D} \longrightarrow \mathbb{R}$
by
\[
\mbox{cov}[a](x_1,x_2):=E\big[\big(a(x_1)-E[a(x_1)]\big)\big(a(x_2)-E[a(x_2)]\big)\big].
\]
For the numerical experiments,  we use a two point exponential covariance  function for the stochastic field  $a$,
\begin{equation}
\label{exp-cov}
\mbox{cov}[a](x_1,y_1;x_2,y_2)=\sigma^2\exp\left( -\frac{\mid x_1-x_2\mid^2}{2l_x^2}-\frac{\mid y_1-y_2\mid^2}{2l_y^2}\right),
\end{equation}
where $(x_i, y_i)$ ($i=1,2$) is the spatial coordinate in 2D,
$\sigma^2$ is the variance of stochastic field $a$, and  $l_x$ and
$l_y$ denote the correlation length in the $x-$ and $y-$direction,
respectively. We will specify these parameters to define the covariance function in numerical examples.
Using  Karhunen-Lo\`{e}ve expansion (KLE)   the random field $a(x, \Theta)$ admits the following decomposition
\begin{equation} \label{KLE}
a(x, \Theta):= E[a]+\sum_{i=1}^N \sqrt{\lambda_i} b_i(x) \theta_i,
\end{equation}
where the random vector $\Theta:=(\theta_1, \theta_2,\cdots,
\theta_N)\in \mathbb{R}^N$ and the random variables
$\{\theta_i(\omega)\}_{i=1}^N$ are mutually orthogonal and have zero
mean and unit variance.
The effects of permeability field
$a(x,\Theta)$ with uniform, beta and Gaussian distributions on the mean
and standard deviation of output were discussed in \cite{Lt10}, where
the numerical results showed similar peak values of standard
deviation for the three different distributions.  Therefore,
in this section we assume for numerical convenience
that each $\theta_i$ is i.i.d. and uniform on $[-1,1]$.
Although, this assumption leads to potentially less physical
fields, it generates variability suitable for evaluating
the performance of the proposed methods.

When MMsFEM is used, the fine grid is coarsened to form a uniform coarse grid.
We solve the pressure equation on the coarse grid using
MMsFEM and then reconstruct the fine-scale velocity field as a superposition
of the multiscale basis functions. The reconstructed velocity field is used to solve the
saturation equation with a finite volume method on the fine grid. We solve the two-phase flow
system (\ref{tp-system}) using the classical IMPES (Implicit Pressure Explicit Saturation).  The temporal discretization is an implicit scheme, which is unconditionally stable but leads to  a nonlinear system (Newton-Raphson
iteration solves the nonlinear system).   In all numerical simulations, mixed multiscale basis functions are constructed
once at the beginning of  computation. In the discussion,
we refer to the grid where multiscale basis functions are constructed
as the coarse grid. Limited global information is computed on the fine grid.

We compare the saturation fields and
water-cut data as a function of pore volume injected (PVI).
The water-cut  is defined as the fraction of water in
the produced fluid and is given by $q_w/q_t$, where $q_t=q_o+q_w$,
with $q_o$ and $q_w$ being the flow rates of oil
and water at the production edge of the model.
In particular, $q_w=\int_{\partial D^{out}} f(S)
{u}\cdot { n} ds$,
$q_t=\int_{\partial \Omega^{out}} {u}\cdot { n} ds$,
where ${\partial D^{out}}$ is the out-flow boundary.
Pore volume injected is defined as $\mbox{PVI}={1 \over V_p} \int_0^t
q_t(\tau) d\tau$, with $V_p$ being the total pore volume of
the system, and provides the dimensionless time for the displacement.
We consider a traditional quarter
five-spot
problem (e.g., \cite{aa}) on a square domain $D$,  where the water is injected at
 left-top corner
and oil is produced at the right-lower corner of the rectangular domain (in the 2D examples).

\subsection{Linear transport}
\label{sect-linear}
In this subsection, we consider the case when the transport (saturation)  equation  is linear in the flow system (\ref{tp-system}).
For this purpose, we consider the mobility of water and mobility of oil defined as following,
\[
\lambda_w(S)=S,  \quad  \lambda_o(S)=1-S.
\]
Consequently, the total mobility $\lambda(S)=1$ and the fractional  flow of water $f_w(S)=S$. Then the two-phase flow system (\ref{tp-system}) reduces to
a linear flow model. Since $\lambda(S)=1$,  the velocity is not updated to compute saturation.  We use the example to investigate the performance of the hybrid
HDMR technique and MMsFEM for linear models.

 The stochastic permeability field $k(x, \Theta)$ is given by  $k(x, \Theta)=\exp\big(a(x,\Theta)\big)$, where $a(x,\Theta)$ is characterized by the covariance function $\mbox{cov}[a]$ defined in (\ref{exp-cov}), whose  parameters are defined as follows: $\sigma^2=1$, $l_x=l_y=0.1$.  We note that the
deceay rate of the eigenvalues in the KLE (\ref{KLE}) expansion depends inversely on the correlation length and is often used as a guideline to determine the number of terms, $N$
\cite{xiu10}.  Here the decay rate is relatively slow  and we
use 80 terms to represent the random field $a(x,\Theta)$ with a relative error
of less than 5\%.
% with $E[a]=1$.
This implies that $k(x, \Theta)$ is defined in an $80$-dimensional random parameter space, i.e., $N=80$.  Fig. $\ref{Fig1.1}$ depicts a realization of the random field $a(x, \Theta)$.
 The stochastic field $k(x, \Theta)$ is defined in a $60\times 60$ fine grid. We choose $6\times 6$ coarse grid and apply MMsFEM to compute velocity.
The time step is taken to be $0.02$ PVI for discretizing temporal variable.

To demonstrate the efficacy of the proposed method to simulate the
stochastic linear two-phase flow system, we compare the performance of
several methods: fine-scale mixed finite element method (MFEM) with
full random dimensions, MMsFEM on full random dimensions, MMsFEM based
on adaptive HDMR, and MMsFEM based on hybrid HDMR. The solution
computed by fine-scale MFEM with the full stochastic space is referred
to as the reference solution in this paper.  To assess the performance
of the various MMsFEM and HDMR combinations, we compute the mean and
standard deviation (std) for the quantifies of interest from the
two-phase flow model, such as saturation and water-cut.  Since the
random field in the example does not have strong non-local features,
the local MMsFEM (L-MMsFEM) generally gives an accurate approximation
\cite{jem10}, and it is used in this example. To choose the most
active random dimensions for adaptive HDMR and hybrid HDMR, we choose
a threshold constant $\zeta=0.9$ and use criteria (\ref{find-J}).
This approach produces $31$ most active dimensions.  The Smolyak
sparse-grid collocation method with level $2$ is used to tackle
the stochastic space.  We evaluate the flow model's outputs at the
collocation points and use the associated weights to compute the mean
and standard deviation of the outputs.  Table $\ref{tab1-sect1}$ lists
the number of deterministic models to be solved for the above four
different methods.  From the table, we see that MFEM on full random
dimensions needs to compute the largest number of deterministic models
(fine-scale models), MMsFEM based on hybrid HDMR needs to compute the
smallest number of deterministic models (coarse-scale models).  This
means that MMsFEM based on hybrid HDMR significantly reduces the
computational complexity.

\begin{table}[hbtp]
\centering \caption{number of deterministic models for the different methods}
\begin{tabular}{|c|c|}
\hline
methods                    &    number of deterministic models    \\
\hline
MFEM on full random dim.   &      12961 fine-scale models\\
\hline
MMsFEM on full random dim   &    12961    coarse-scale models\\
\hline
MMsFEM based on adaptive HDMR &     6446   coarse-scale models\\
\hline
MMsFEM based on hybrid HDMR      &     2231    coarse-scale models\\
\hline
\end{tabular}
\label{tab1-sect1}
\end{table}

Fig. \ref{Fig1.2} shows the point-wise mean of the saturation field at
$\mbox{PVI}=0.4$.  In this figure, we observe that both L-MMsFEM based
on the hybrid HDMR and L-MMsFEM based on the adaptive HDMR are almost
identical to the reference mean computed by fine-scale MFEM on full
random dimensions.  Fig. \ref{Fig1.3} illustrates the point-wise
standard deviation of saturation at $\mbox{PVI}=0.4$ and leads to four
observations: (1) Compared with reference standard deviation, the
approximations of standard deviation by the three multiscale models
(L-MMsFEM on full random dim., L-MMsFEM based on hybrid HDMR and
L-MMsFEM based on adaptive HDMR) are good; (2) L-MMsFEM on full random
dim. renders better approximation for standard deviation than L-MMsFEM
based on truncated HDMR techniques; (3) L-MMsFEM based on hybrid HDMR
gives almost the same standard deviation as the L-MMsFEM based on
adaptive HDMR; (4) The variance mostly occurs along the advancing
water front, and where the front interacts with the domain boundary.
Fig. \ref{Fig1.4} shows the mean of water-cut curves, which are all
nearly identical. This demonstrates that a good approximation has been
achieved.  Fig. \ref{Fig1.5} shows the standard deviation of water-cut
curves from the four different methods.  Here, we see that both
L-MMsFEM based on hybrid HDMR and L-MMsFEM based on adaptive HDMR
produce a very good approximation for the standard deviation at most
time instances.  The standard deviation of the water-cut curve from
L-MMsFEM based on hybrid HDMR is almost identical to the standard
deviation of water-cut curve from L-MMsFEM based on adaptive HDMR.

Next we discuss the relative errors of the domain integrated mean and
standard deviation.  Let $S_{r}$/$W_r$, $S_{ms, f}$/$W_{ms,f}$,
$S_{ms, aH}$/$W_{ms, aH}$ and $S_{ms, hH}$/$W_{ms, hH}$ be the
saturation/water-cut using MFEM on full random dimensions, MMsFEM on
full random dimensions, MMsFEM based on adaptive HDMR and MMsFEM based
on hybrid HDMR, respectively. Then we define the relative errors of
saturation mean and saturation standard deviation between $S_{r}$ and
$S_{ms, f}$ as following,
\begin{equation}
\label{relative-error-S}
\mathcal{E}_{ms}^{m}(S)=\frac{\|E[S_r](x)-E[S_{ms, f}](x)\|_{L^1(D)}}{\|E[S_r](x)\|_{L^1(D)}},  \quad
\mathcal{E}_{ms}^{std}(S)=\frac{\|\sigma[S_r](x)-\sigma[S_{ms, f}](x)\|_{L^1(D)}}{\|\sigma[S_r](x)\|_{L^1(D)}},
\end{equation}
where $\sigma$ denotes the standard deviation operation.  We can similarly define the relative errors $\mathcal{E}_{aH}^{m}(S)$  and $\mathcal{E}_{aH}^{std}(S)$  between $S_{r}$ and $S_{ms, aH}$,
and the relative errors $\mathcal{E}_{hH}^{m}(S)$ and $\mathcal{E}_{hH}^{std}(S)$ between $S_{r}$ and $S_{ms, hH}$.
We list these errors on saturation at $0.4$ PVI in Table $\ref{tab2-sect1}$. From the table, we see that L-MMsFEM based on hybrid HDMR gives the same accuracy of mean as L-MMsFEM based on adaptive HDMR. They are both very close to
the error of mean by L-MMsFEM on full random dimensions. The relative error of standard deviation by L-MMsFEM based on hybrid HDMR also has a good agreement with that of L-MMsFEM based on adaptive HDMR.
 The table implies  that the errors only from   adaptive HDMR and  hybrid HDMR are very small. The main source of errors is from spatial multiscale approximation.
\begin{table}[hbtp]
\centering \caption{relative errors of mean and std. on saturation (at 0.4 PVI) }
\begin{tabular}{|c|c|c|}
\hline
methods                    &    relative error of mean    & relative error of std.    \\
\hline
L-MMsFEM on full random dim.   &   1.223540e-002                 &   5.302182e-002     \\
\hline
L-MMsFEM based on adaptive HDMR &  1.361910e-002                &  9.389291e-002        \\
\hline
L-MMsFEM based on hybrid HDMR      &   1.361910e-002               &   9.482789e-002       \\
\hline
\end{tabular}
\label{tab2-sect1}
\end{table}

For water-cut, the relative errors of  mean and  standard deviation between $W_{r}$ and $W_{ms, f}$ are defined by
\begin{equation}
\label{relative-error-W}
\mathcal{E}_{ms}^{m}(W)=\frac{\|E[W_r](t)-E[W_{ms, f}](t)\|_{L^2([0, 1]}}{\|E[W_r](t)\|_{L^2([0, 1])}},  \quad
\mathcal{E}_{ms}^{std}(W)=\frac{\|\sigma[W_r](t)-\sigma[W_{ms, f}](t)\|_{L^2([0,1])}}{\|\sigma[W_r](t)\|_{L^2([0,1])}},
\end{equation}
where $[0,1]$ is the temporal domain. The relative errors of water-cut, such as  $\mathcal{E}_{aH}^{m}(W)$, $\mathcal{E}_{aH}^{std}(W)$, $\mathcal{E}_{hH}^{m}(W)$ and $\mathcal{E}_{hH}^{std}(W)$,
are defined in a similar way. We list these relative errors in Table \ref{tab3-sect1}.  The table shows similar behavior for these methods as was observed in the calculation of saturation mean and saturation standard deviation.  The performance
of L-MMsFEM based on hybrid HDMR  is almost the same as the performance of L-MMsFEM based on adaptive HDMR.
\begin{table}[hbtp]
\centering \caption{relative errors of mean and std. on water-cut}
\begin{tabular}{|c|c|c|}
\hline
methods                    &    relative error of mean    & relative error of std.    \\
\hline
L-MMsFEM on full random dim.   &   1.445791e-002                  &   4.339328e-002      \\
\hline
L-MMsFEM based on adaptive HDMR &  1.856908e-002                  &   1.006136e-001         \\
\hline
L-MMsFEM based on hybrid HDMR      &   1.856908e-002                 &   1.040452e-001        \\
\hline
\end{tabular}
\label{tab3-sect1}
\end{table}

\begin{figure}[tbp]
\centering
\includegraphics[width=5in, height=2in]{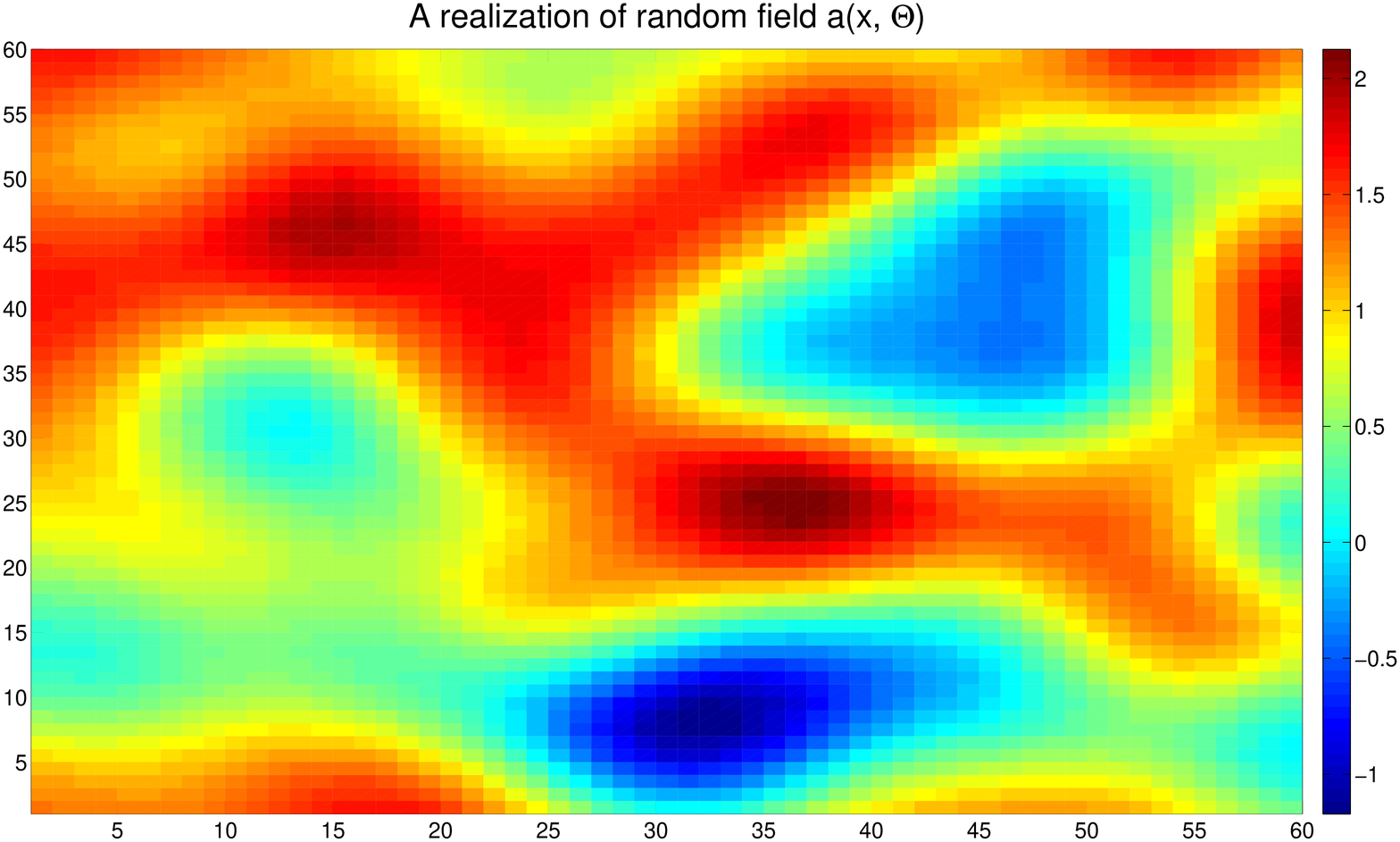}
\caption{A realization of random field $a(x, \Theta)$}
\label{Fig1.1}
\end{figure}

\begin{figure}[tbp]
\centering
\includegraphics[width=6in, height=4in]{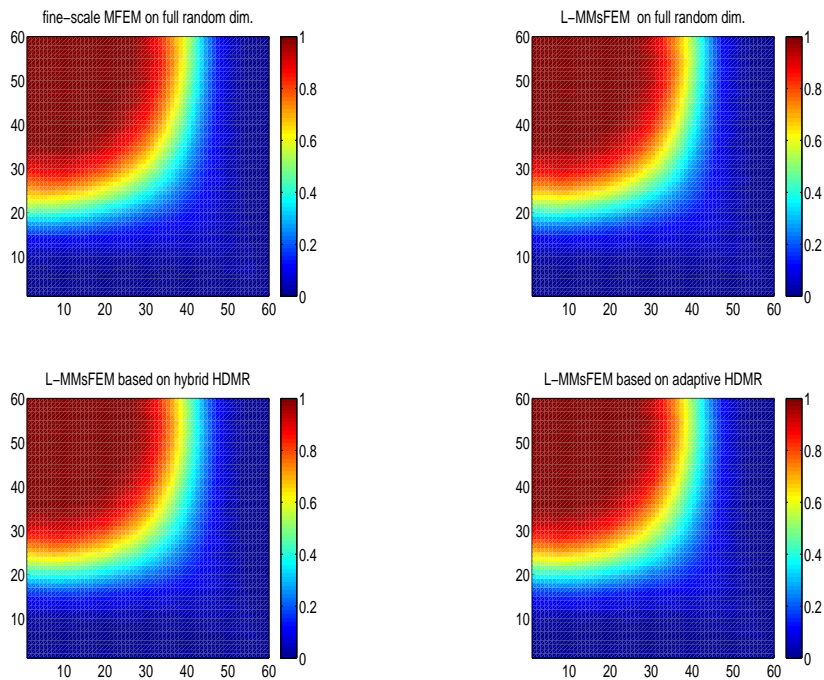}
\caption{Point-wise mean of saturation map at PVI=0.4. Top left: MFEM on full random dimensions; Top right: L-MMsFEM on full random dimensions; Bottom left: L-MMsFEM based on hybrid HDMR;
Bottom right: L-MMsFEM based on adaptive HDMR}
\label{Fig1.2}
\end{figure}

\begin{figure}[tbp]
\centering
\includegraphics[width=6in, height=4in]{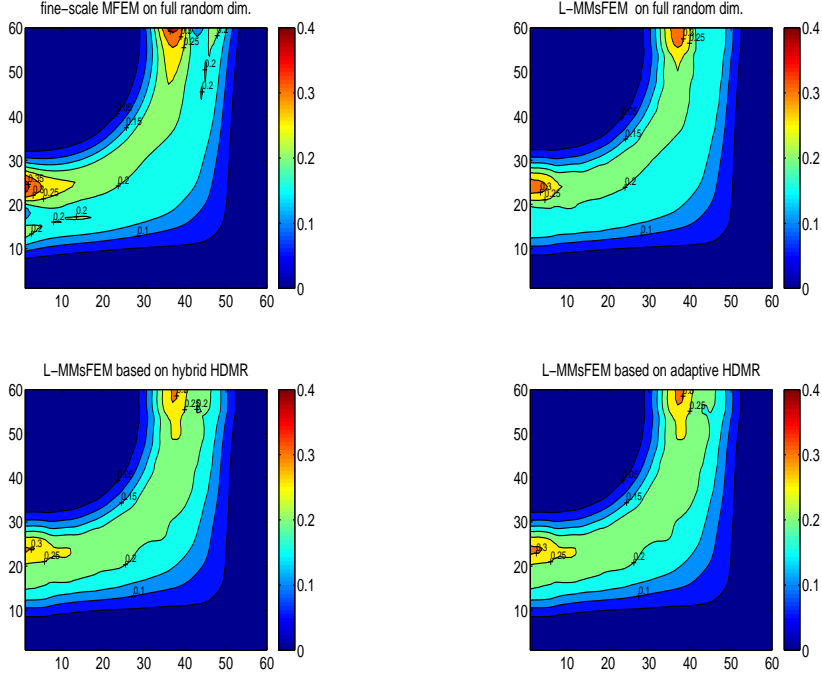}
\caption{Point-wise standard deviation of saturation contour  at PVI=0.4. Top left: MFEM on full random dimensions; Top right: L-MMsFEM on full random dimensions; Bottom left: L-MMsFEM based on hybrid HDMR;
Bottom right: L-MMsFEM based on adaptive HDMR}
\label{Fig1.3}
\end{figure}

\begin{figure}[tbp]
\centering
\includegraphics[width=5in, height=2.5in]{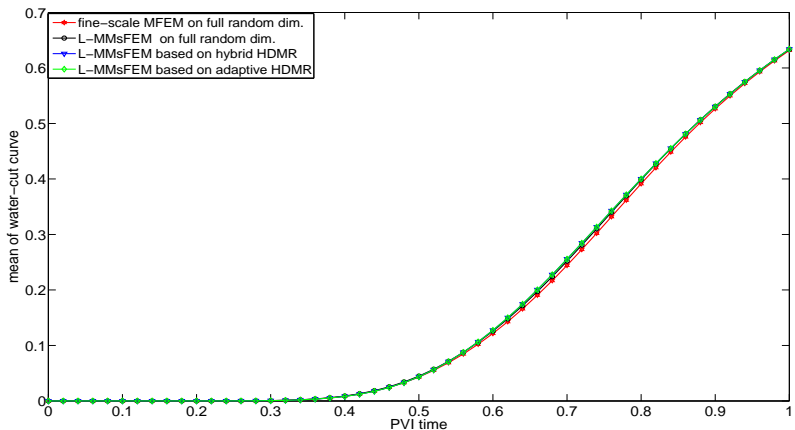}
\caption{Mean of water-cut curves}
\label{Fig1.4}
\end{figure}

\begin{figure}[tbp]
\centering
\includegraphics [width=5in, height=2.5in]{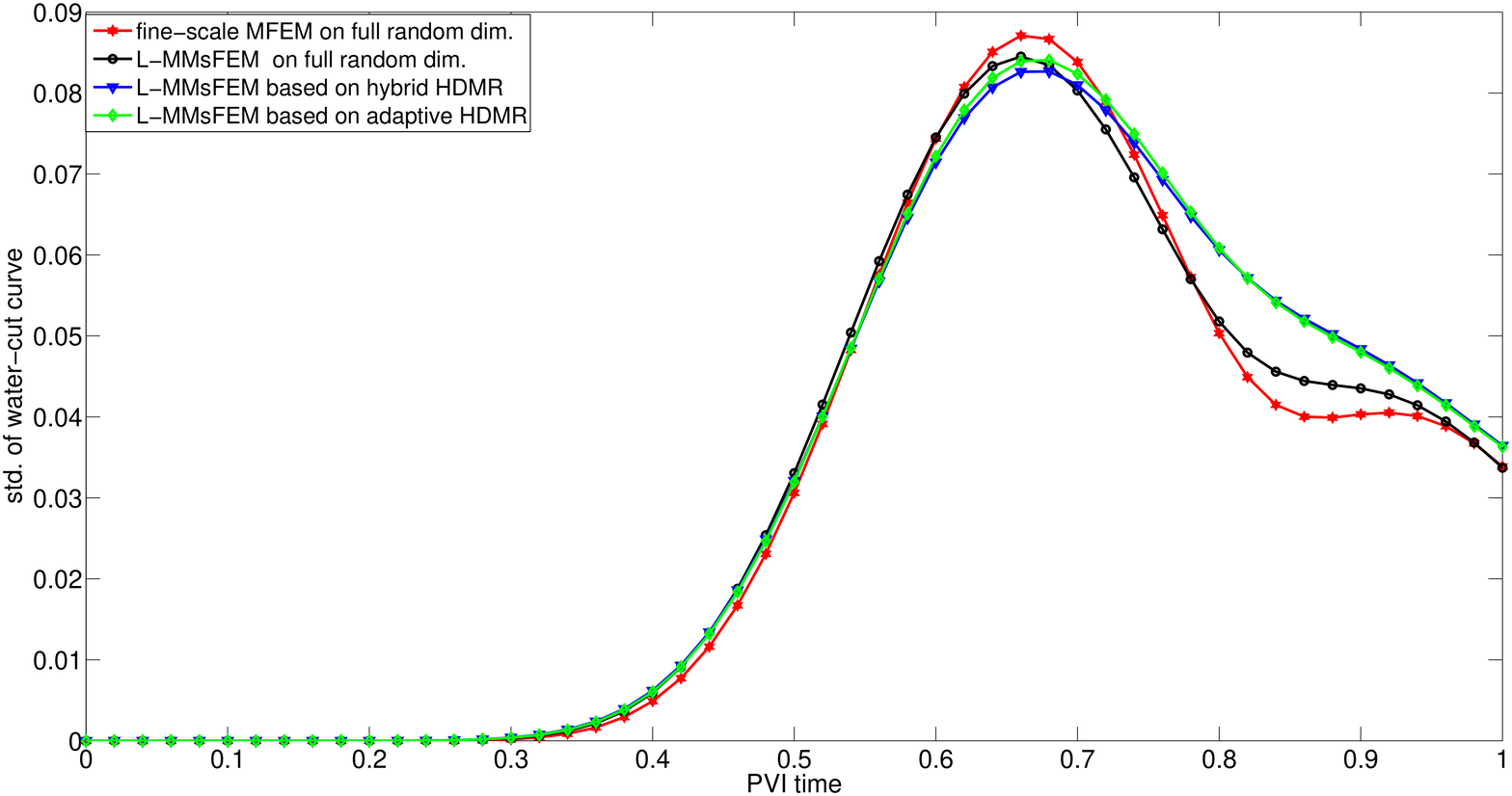}
\caption{Standard deviation  of water-cut curves}
\label{Fig1.5}
\end{figure}

The number of the most active dimensions may have an important impact
on both the computational efficiency and solution accuracy.  To this
end, we investigate the performance for different threshold constants
in (\ref{find-J}).  Specifically, we take $\zeta=0.75$, $0.8$, $0.85$,
$0.9$, $0.95$.  For these threshold constants, we find that the
corresponding number of most active dimensions are
$dim(\Theta_{\mathcal{J}})=19$, $22$, $26$, $31$, $41$.
Fig.~\ref{Fig1.6} displays the CPU times of L-MMsFEM based on adaptive
HDMR and L-MMsFEM based on hybrid HDMR for this set of active
dimensions. Here, two observations are important to note: (1) The CPU
time of L-MMsFEM based on hybrid HDMR is only a fraction of the CPU
time using L-MMsFEM based on adaptive HDMR; (2) As the number of the
most active dimensions increases, the CPU time of L-MMsFEM based on
hybrid HDMR increases mildly, but the CPU time of L-MMsFEM based on
adaptive HDMR increases dramatically.  This comparison demonstrates
that hybrid HDMR is much more efficient than adaptive HDMR.  There are
two reasons that adaptive HDMR requires more CPU time: (1) adaptive
HDMR needs to solve more deterministic models; (2) to compute the
standard deviation (or variance), a large number of stochastic
interpolations are involved.  Fig. \ref{Fig1.7} shows the relative
errors of mean (left) and standard deviation (right) for saturation at
$0.4$ PVI.  From the figure, we conclude that increasing the number of
the most active dimensions can substantially improve the accuracy of
the mean and standard deviation for saturation.  We also see that
for the saturation approximation the
L-MMsFEM based on hybrid HDMR has almost the same accuracy as the
L-MMsFEM based on adaptive HDMR.
Fig. \ref{Fig1.8} presents the relative errors of water-cut mean
(left) and water-cut standard deviation (right) for the different
number of the most active dimensions.  From the figure, we observe
that the reduction in the error with increasing $dim(\Theta_{\mathcal{J}})$ is
smaller for water-cut than for mean saturation.  Here,
the error from L-MMsFEM is dominating the total error, and hence,
is overwhelming the impact of improved accuracy in the HDMR itself.

\begin{figure}[tbp]
\centering
\includegraphics [width=5in, height=2.5in]{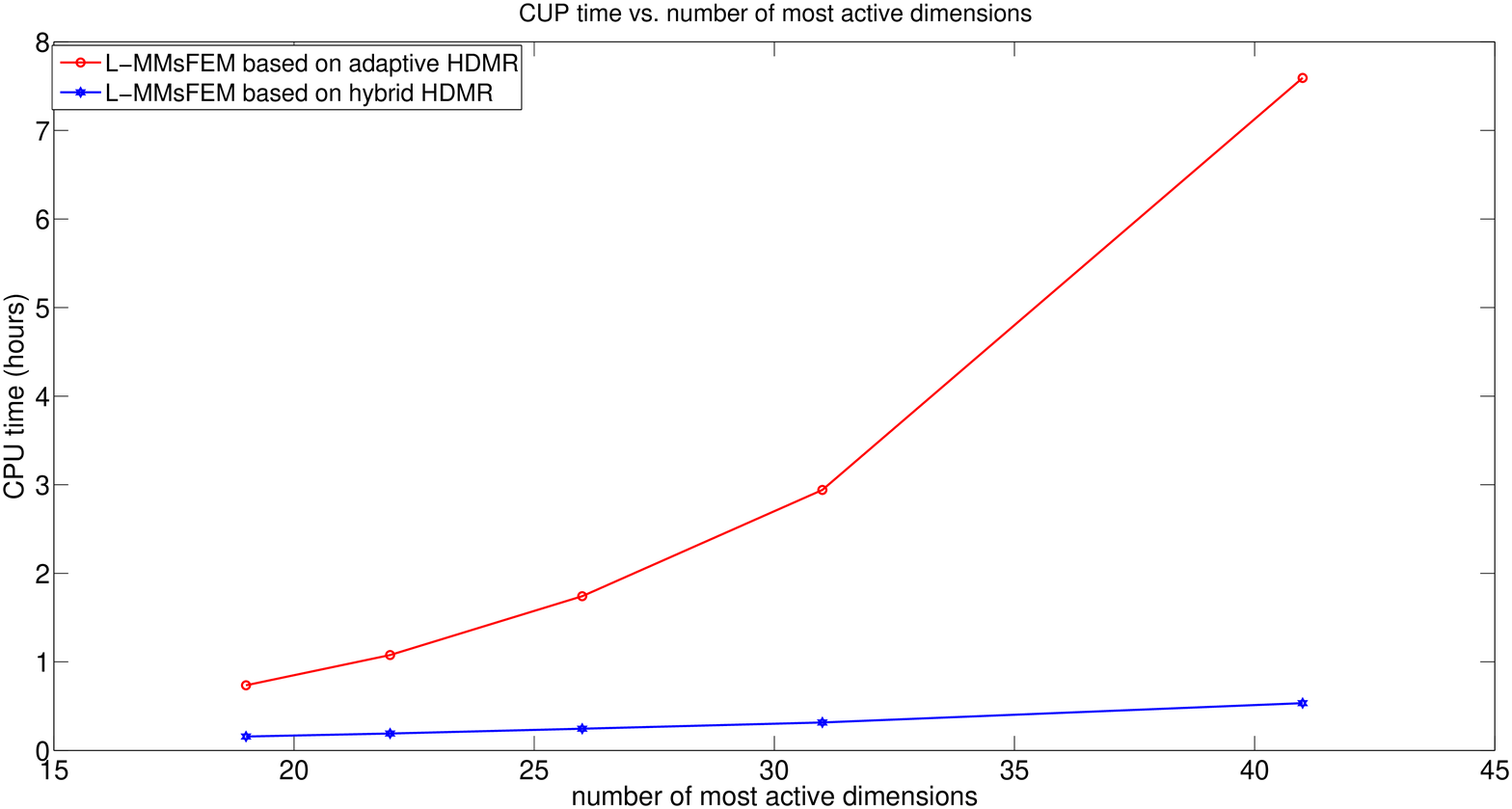}
\caption{CPU time for the different number of the most active dimensions, $19, 22,26,31,41$.}
\label{Fig1.6}
\end{figure}

\begin{figure}[tbp]
\centering
\includegraphics [width=5in, height=2.5in]{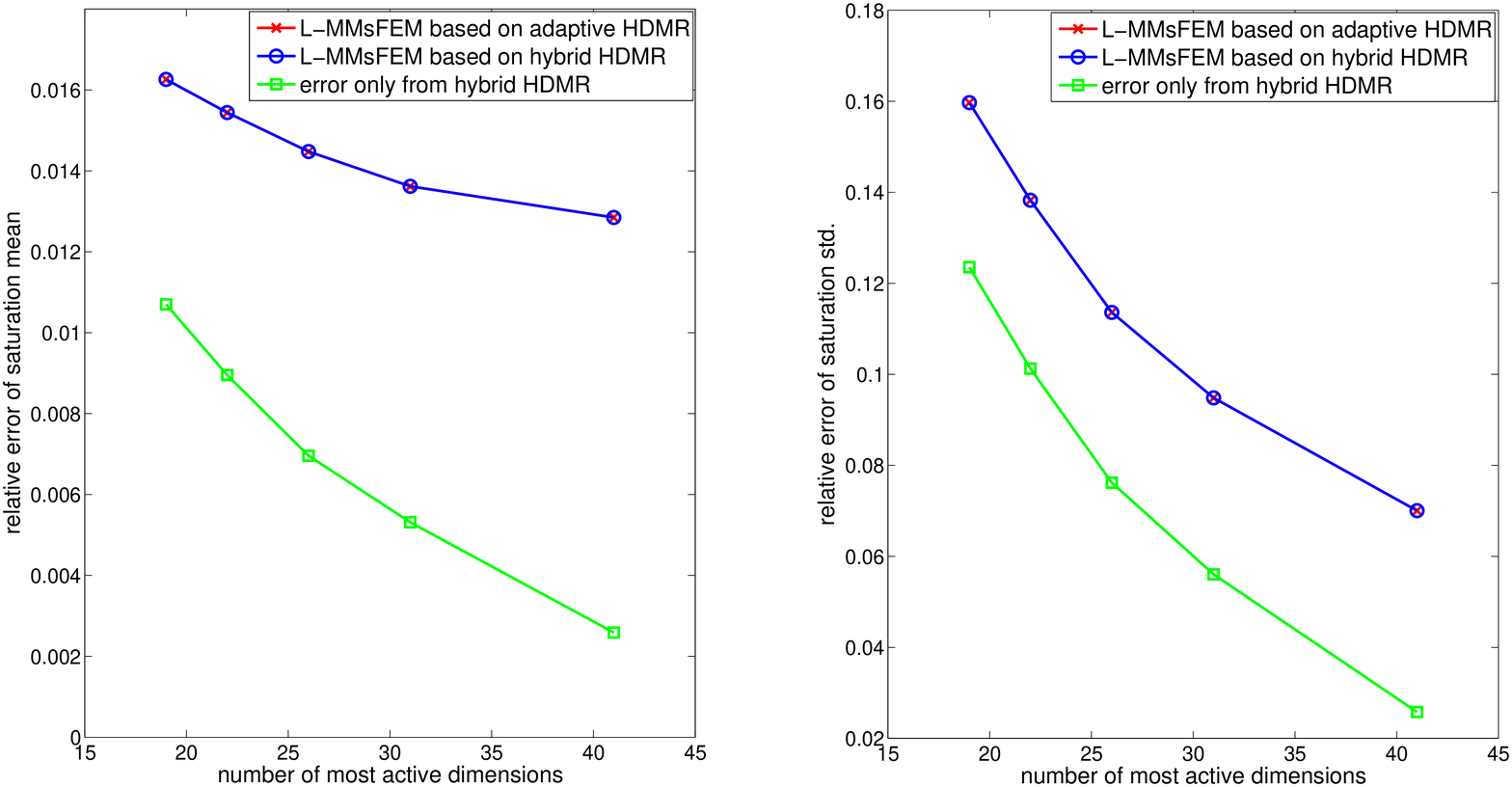}
\caption{Relative errors of mean (left) and standard deviation (right) for saturation at $0.4$ PVI.}
\label{Fig1.7}
\end{figure}

\begin{figure}[tbp]
\centering
\includegraphics [width=5in, height=2.5in]{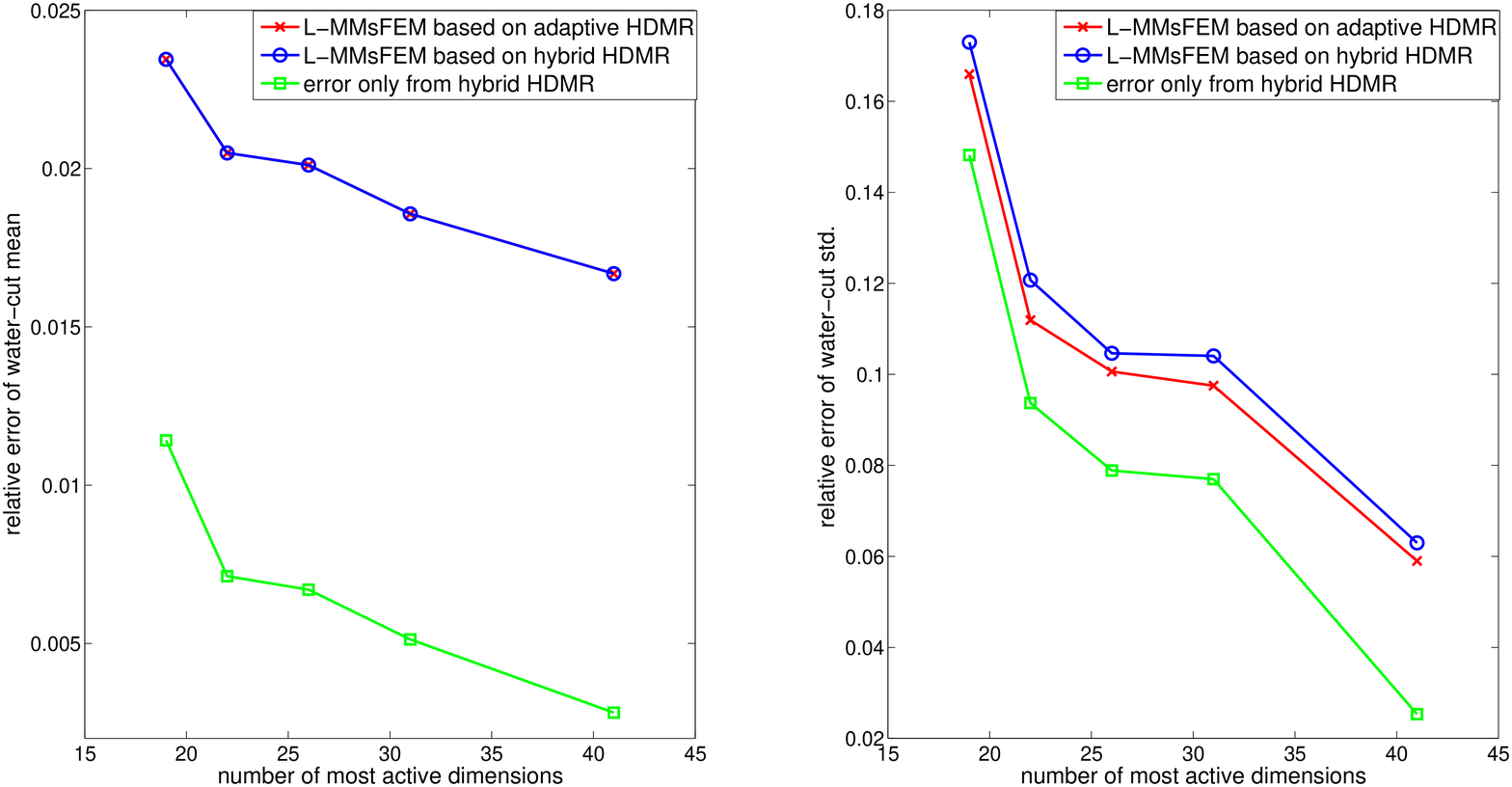}
\caption{Relative errors of mean (left) and standard deviation (right) for water-cut.}
\label{Fig1.8}
\end{figure}

\subsection{Non-linear transport}
In this subsection, we consider the two-phase flow system
(\ref{tp-system}) with the mobilities of water and oil defined by
nonlinear functions of saturation,
\[
\lambda_w(S)=S^2/\mu_w, \quad  \lambda_o(S)=(1-S)^2/\mu_o.
\]
Here $\mu_w/\mu_o=0.1$, the ratio between viscosity of water and
oil. Consequently, the fractional flow function $f_w(S)$ of water is given by
\[
f_w(S)=\frac{S^2}{S^2+0.1(1-S)^2}.
\]
This results in a non-linear two-phase flow system.

We again consider the random permeability
$k(x,\Theta)=\exp\big(a(x,\Theta)\big)$. Here the covariance function
$\mbox{cov}[a]$ of $a(x,\Theta)$ is defined in (\ref{exp-cov}) with
$l_x=l_y=0.2$ and $\sigma^2=1$.  The mean of $a(x,\Theta)$ is highly
heterogeneous and its map is depicted in Fig. \ref{Fig2.1} (left). It
is actually obtained by extracting and rescaling an SPE 10 \cite{cb01}
permeability field (the $45$-th layer).  We truncate the KLE
(\ref{KLE}) after the first $30$ terms to represent the random field
$a(x,\Theta)$ and so the random field $k(x, \Theta)$ is defined in a
$30$-dimensional random parameter space. Fig. $\ref{Fig2.1}$ shows a
realization of the random field $a(x, \Theta)$ (right).  The
stochastic field $k(x,\Theta)$ is defined on a $60\times 60$ fine
grid. We choose a $6\times 6$ coarse grid for MMsFEM to compute the
velocity.  To discretize the temporal variable of the saturation
equation, the time step is taken to be $0.01$ PVI.

Fig. \ref{Fig2.1} shows that the permeability field $k(x, \Theta)$
exhibits some channelized features, which have important an impact on
the flow.  To achieve accurate simulation results in this setting,
limited global information can improve the accuracy
\cite{aa,aej08,jem10}. In this subsection, we incorporate the hybrid HDMR
technique and the adaptive HDMR technique with both local MMsFEM
(L-MMsFEM) and global MMsFEM (G-MMsFEM), and compare the performance
of these methods. To obtain the most active random dimensions for the
truncated HDMR techniques, we choose a threshold constant $\zeta=0.9$
and use criteria (\ref{find-J}).  This gives rise to $14$ most
active dimensions in the $30$-dim random parameter space.  We use the
collocation points and weights associated with level $2$ Smolyak
sparse-grid collocation to compute the mean and standard deviation.

Fig. \ref{Fig2.2} shows the point-wise mean of saturations at $0.4$
PVI for different methods.  We can clearly see that G-MMsFEM on full
random dimensions yields the best approximation to the reference
solution, which is given by fine-scale MFEM on full random dimensions.
The figure also shows that the multiscale methods with the truncated
HDMR techniques provide good approximations.  Fig. \ref{Fig2.3}
describes the point-wise standard deviation of saturation at $0.4$ PVI
for those different methods. Based on the figure, we make two
observations: (1) the largest variance occurs around the water front;
(2) the hybrid HDMR technique gives almost the same standard deviation
map as the adaptive HDMR technique.  We note that these observations
are consistent with those of Section~\ref{sect-linear}.  We also
compute the water-cut curves for the different methods.  To minimize
overlap in the visualization, we only present the water-cut curves
for four methods: MFEM on full random dim., G-MMsFEM on full random dim.,
G-MMsFEM based on adaptive HDMR and G-MMsFEM based on hybrid HDMR.
Fig. \ref{Fig2.4} shows the mean of the water-cut curves for the four
methods. Here, the four curves overlap each other at almost all
times. We note that there exists a small fluctuation right after
the water breaks through.  The reason may be that the value
of water-cut changes sharply right after water break-through time.
The standard deviation of the water-cut curves for the four methods
are illustrated in Fig. \ref{Fig2.5}.  From the figure, we find that the
variance of water-cut rises rapidly at the break-through time,
and has a small number of peaks that are likely related to the
dominant channelized flow paths in the underlying mean permeability
field.

\begin{figure}[tbp]
\centering
\includegraphics[width=6in, height=2in]{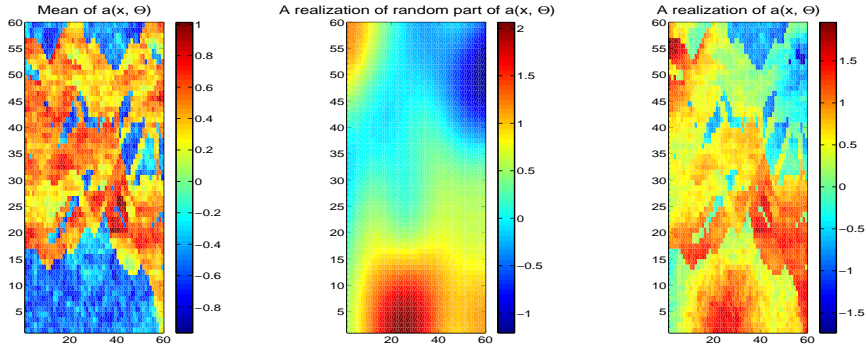}
\caption{A realization of random field $a(x, \Theta)$}
\label{Fig2.1}
\end{figure}

\begin{figure}[tbp]
\centering
\includegraphics[width=6in, height=3in]{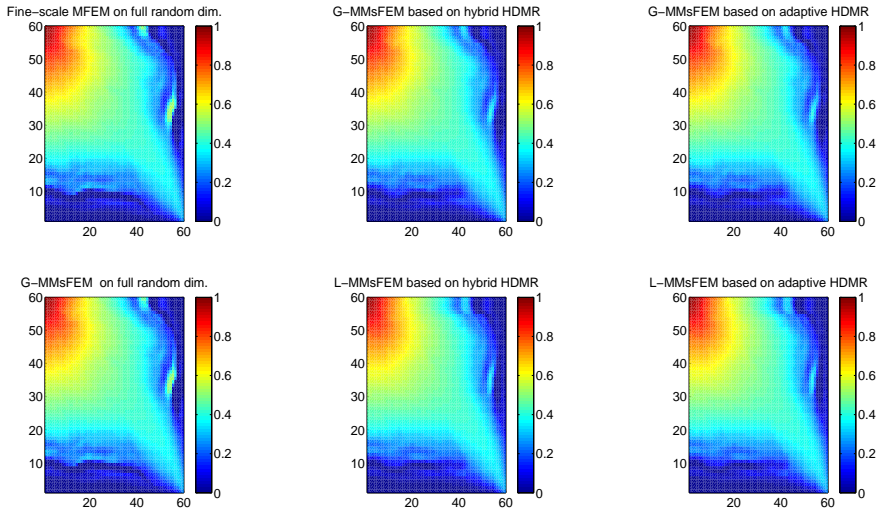}
\caption{Point-wise mean of saturations at $0.4$ PVI for different methods}
\label{Fig2.2}
\end{figure}

\begin{figure}[tbp]
\centering
\includegraphics[width=6in, height=3in]{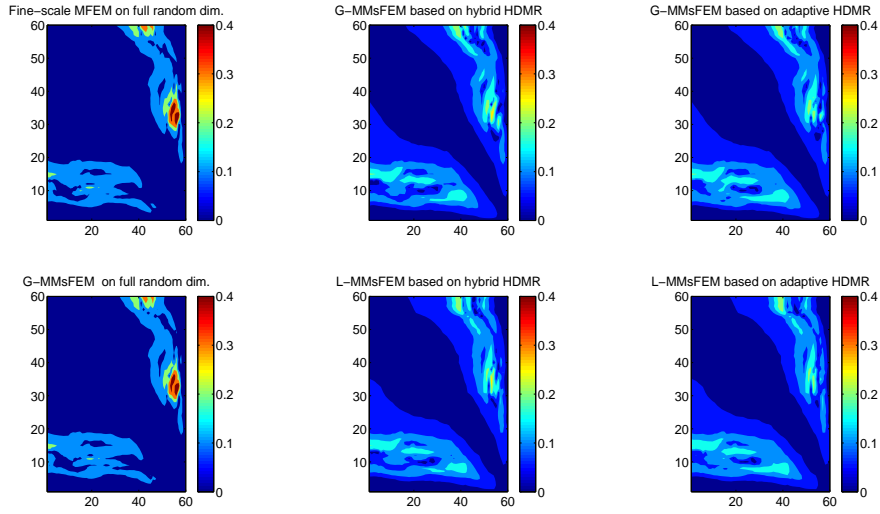}
\caption{Point-wise standard deviation of saturations at $0.4$ PVI for different methods}
\label{Fig2.3}
\end{figure}

\begin{figure}[tbp]
\centering
\includegraphics[width=5.5in, height=2in]{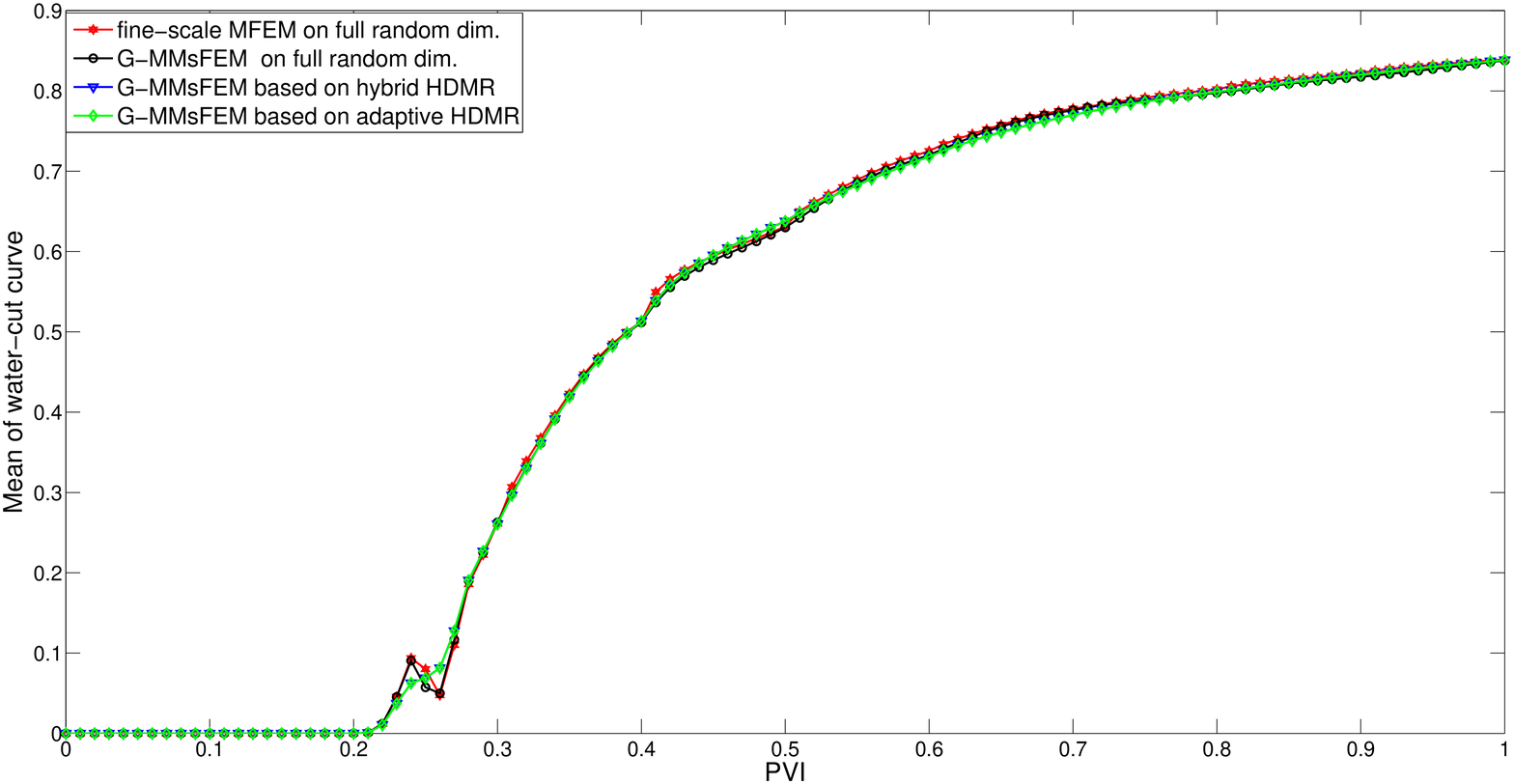}
\caption{Mean  of water-cut curves  for MFEM on full random dim., G-MMsFEM on full random dim., G-MMsFEM based on adaptive HDMR and G-MMsFEM based on hybrid HDMR}
\label{Fig2.4}
\end{figure}

\begin{figure}[tbp]
\centering
\includegraphics[width=5.5in, height=2in]{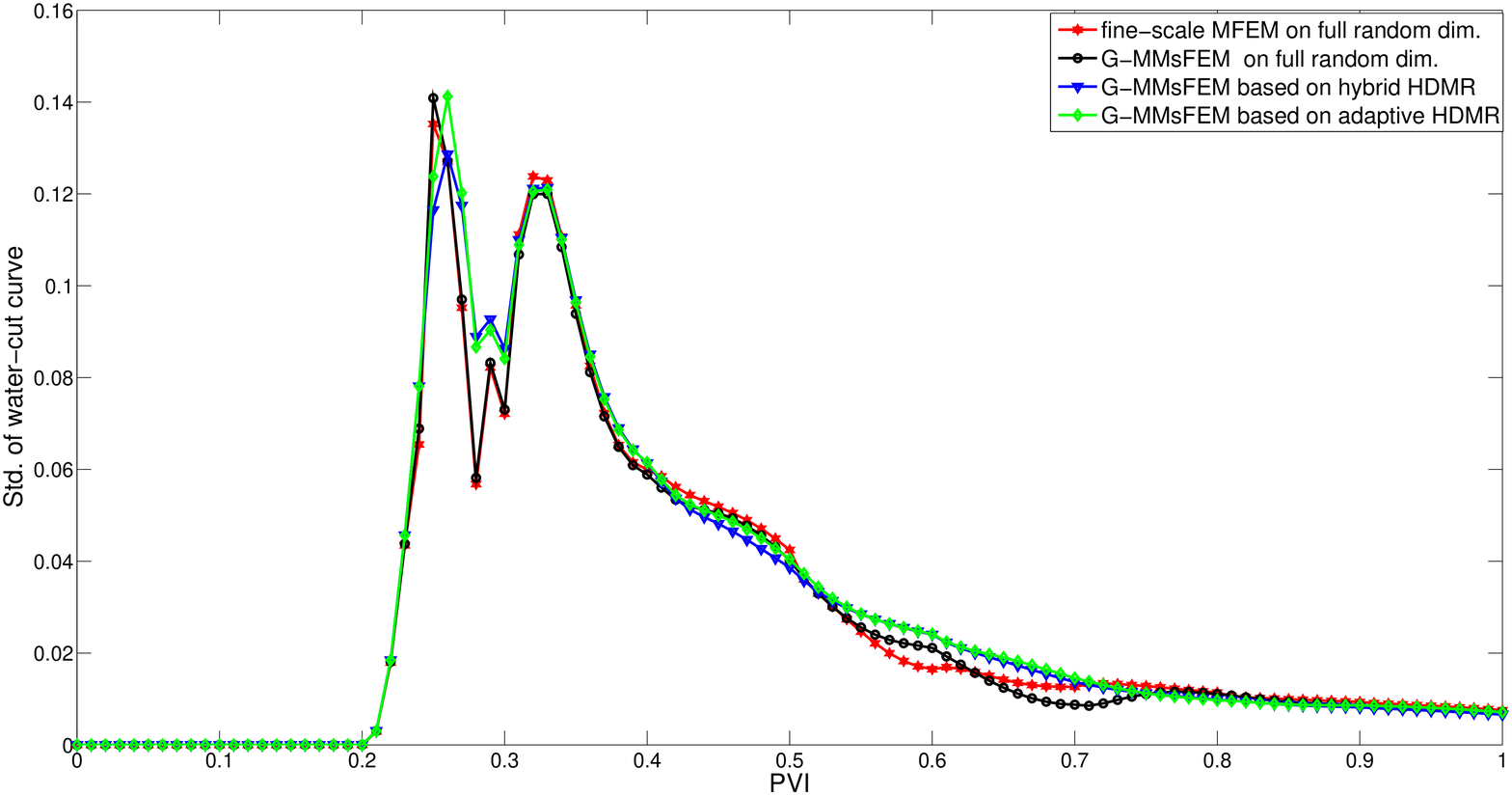}
\caption{Standard deviation  of water-cut curves  for MFEM on full random dim., G-MMsFEM on full random dim., G-MMsFEM based on adaptive HDMR and G-MMsFEM based on hybrid HDMR}
\label{Fig2.5}
\end{figure}

In order to carefully measure the differences caused by L-MMsFEM and
G-MMsFEM when integrated with adaptive HDMR and hybrid HDMR, we follow
the procedure outlined in Subsection \ref{sect-linear} and compute the
relative errors between the reference solution and the solutions of
the various coarse models (MMsFEM with (or without) HDMR techniques).
As before, the reference solution is solved by fine-scale MFEM on full
random dimensions.  The relative errors of saturation and water-cut
are defined similarly to (\ref{relative-error-S}) and
(\ref{relative-error-W}), respectively.  Table \ref{tab1-sect2} and
\ref{tab2-sect2} list the relative errors of mean and standard
deviation for saturation at $0.4$ PVI and water-cut,
respectively. Examining this table, we note: (1) limited global
information can enhance the accuracy of MMsFEM; (2) the results by
adaptive and hybrid HDMR are very close to each other.

\begin{table}[hbtp]
\centering \caption{relative errors of mean and std. on saturation (at 0.4 PVI) }
\begin{tabular}{|c|c|c|}
\hline
methods                    &    relative error of mean    & relative error of std.    \\
\hline
L-MMsFEM on full random dim.   &              4.170726e-002  &  1.545560e-001       \\
\hline
G-MMsFEM on full random dim.   &                 8.950902e-003  &  4.027147e-002        \\
\hline
L-MMsFEM based on adaptive HDMR &                 4.788991e-002  &  2.051337e-001         \\
\hline
G-MMsFEM based on adaptive HDMR &                3.283281e-002  &  1.608348e-001       \\
\hline
L-MMsFEM based on hybrid HDMR      &              4.788991e-002  &  2.013145e-001   \\
\hline
G-MMsFEM based on hybrid HDMR      &                3.283281e-002  &  1.520236e-001    \\
\hline
\end{tabular}
\label{tab1-sect2}
\end{table}

\begin{table}[hbtp]
\centering \caption{relative errors of mean and std. on water-cut}
\begin{tabular}{|c|c|c|}
\hline
methods                    &    relative error of mean    & relative error of std.    \\
\hline
L-MMsFEM on full random dim.   &      1.320466e-002  &  8.836354e-002                          \\
\hline
G-MMsFEM on full random dim.   &          8.000854e-003  &  4.298053e-002                       \\
\hline
L-MMsFEM based on adaptive HDMR &           1.423242e-002  &  1.352706e-001                    \\
\hline
G-MMsFEM based on adaptive HDMR &             1.162734e-002  &  1.212131e-001                  \\
\hline
L-MMsFEM based on hybrid HDMR      &             1.423242e-002  &  1.447519e-001                   \\
\hline
G-MMsFEM based on hybrid HDMR      &                1.162734e-002  &  1.248755e-001                \\
\hline
\end{tabular}
\label{tab2-sect2}
\end{table}

Finally we examine the efficiency for the different approaches. To
this end, we record the CPU time  for the seven different methods:
fine-scale MFEM on full random dimensions, L-MMsFEM on full random
dimensions, G-MMsFEM on full random dimensions, L-MMsFEM based on
adaptive HDMR, G-MMsFEM based on adaptive HDMR, L-MMsFEM based on
hybrid HDMR and G-MMsFEM based on hybrid HDMR.  Fig. \ref{Fig2.6}
shows the CPU time for the seven different approaches.  Here it is
apparent that the approaches based on hybrid HDMR need the least CPU
time and achieve the best efficiency.  The CPU time of hybrid HDMR is
only a fraction, about a third, of the CPU time of adaptive HDMR.
Moreover, this speedup was shown earlier to increase quickly with the
number of most active dimensions.  Comparing the performance of the
local and global MMsFEM based methods, we find that for each HDMR
technique, the CPU time  of G-MMsFEMs is  only slightly larger than
the CPU time of L-MMsFEM.  For adaptive HDMR, a significant amount of
CPU time is spent in computing variance.  However, computation of
the variance in hybrid HDMR is straightforward and requires very little
CPU time. This is an important advantage of the hybrid HDMR approach.

\begin{figure}[tbp]
\centering
\includegraphics[width=0.9\linewidth,height=3in]{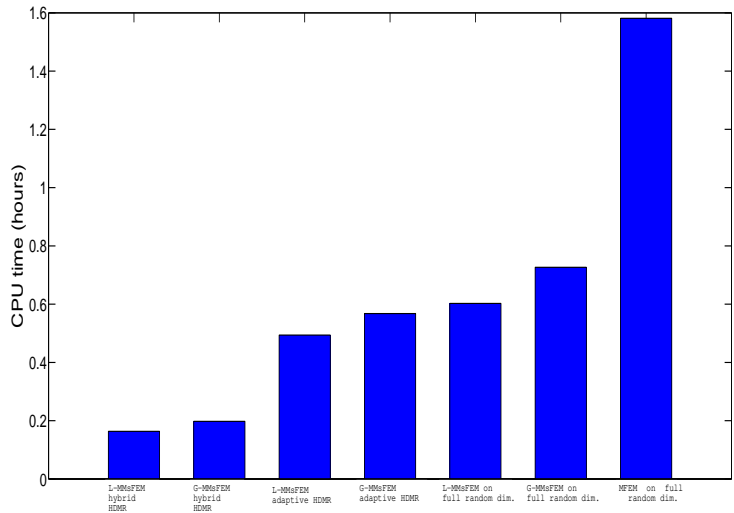}
\caption{CPU time for different approximation models, using 14 most active
  dimensions for the truncated HDMR techniques.}
\label{Fig2.6}
\end{figure}

\section{Conclusions}
In this paper, we presented a general framework for high-dimensional
model representations (HDMRs), and proposed a hybrid HDMR technique in
combination with a mixed multiscale finite element method (MMsFEM) to
simulate two-phase flow through heterogeneous porous media.  The hybrid
HDMR technique decomposes a high-dimensional stochastic model into a
moderate-dimensional stochastic model and a few one-dimensional
stochastic models.  An optimization criteria was developed that
ensures a specified percentage of the total variance is present in the
most active dimensions (i.e., the moderate-dimensional stochastic
model).  In addition, we demonstrated that combining an MMsFEM with
hybrid HDMR could significantly reduce the original model's complexity
in both the resolution of the physical space and the high-dimensional
stochastic space.  We also presented MMsFEM based on adaptive HDMR,
which has been widely used in stochastic model reduction.  Compared
with adaptive HDMR, the hybrid HDMR is much more efficient and retains
the same (or better) accuracy. To capture strong non-local features in
the multiscale models, we have incorporated important global
information into the multiscale computation. This can improve the
approximation accuracy of the proposed coarse multiscale models.  We
carefully analyzed the proposed MMsFEM using HDMR techniques and
discussed both the computational efficiency and approximation errors.
The MMsFEM based on HDMR techniques was applied to two-phase flows in
heterogeneous random porous media. Both linear and non-linear
saturation dependencies of the mobility were considered.  The
simulation results confirmed the performance of the proposed
approaches.

\section*{Acknowledgments}
We thank the reviewers for their insightful comments and suggestions
that helped improve the paper.

\appendix
\section{proof of Theorem  \ref{HDMR-thm3}}
\label{app1}
For the proof, we need to  calculate  the HDMR components in $f^{anova}(\Theta)$.  By definition, we derive the term $\hat{f}^{anova}(\Theta_{\mathcal{J}})$ as follows,
\begin{eqnarray}
\label{eq1-thm2.4}
\begin{split}
\hat{f}^{anova}(\Theta_{\mathcal{J}})&=\hmP_{\mathcal{J}} f^{cut}=\int_{I^{N-J}} f^{cut}(\Theta)\Pi_{i\notin \mathcal{J}} d\mu_i(\theta_i)\\
&=\int_{I^{N-J}} \hat{f}^{cut}(\Theta_{\mathcal{J}})\Pi_{i\notin \mathcal{J}} d\mu_i(\theta_i) +\sum_{i\in \{1,\cdots, N\}\setminus \mathcal{J}}\int_{I^{N-J}} f_i^{cut}(\theta_i)\Pi_{i\notin \mathcal{J}} d\mu_i(\theta_i)\\
&=\hat{f}^{cut}(\Theta_{\mathcal{J}})+\sum_{i\in \{1,\cdots, N\}\setminus \mathcal{J}} E[f_i^{cut}].
\end{split}
\end{eqnarray}
We define $f_0^{anova}=E[f^{cut}]$. Then for any $i\in \{1,\cdots, N\}\setminus \mathcal{J}$,  we have
\begin{eqnarray}
\label{eq2-thm2.4}
\begin{split}
&f_i^{anova}(\theta_i)=\int_{I^{N-1}} f^{cut}(\Theta)  \Pi^N_{\substack{s=1\\ s\neq i}} d\mu_s(\theta_s)-f_0^{anova}\\
&=\int_{I^{N-1}} \hat{f}^{cut}(\Theta_{\mathcal{J}})  \Pi^N_{\substack{s=1\\ s\neq i}} d\mu_s(\theta_s)
+\sum_{j\in \{1,\cdots, N\}\setminus \mathcal{J}}
\int_{I^{N-1}} f_j^{cut}(\theta_j)  \Pi^N_{\substack{s=1\\ s\neq i}} d\mu_s(\theta_s)-f_0^{anova}\\
&=f_i^{cut}(\theta_i)+  E[\hat{f}^{cut}]+\sum_{\substack{j\in \{1,\cdots, N\}\setminus \mathcal{J}\\ j\neq i}}E[f_j^{cut}]-f_0^{anova}\\
&:=f_i^{cut}(\theta_i)+C_i.
\end{split}
\end{eqnarray}
It is obvious that $E[f_i^{anova}]=0$  by (\ref{eq2-thm2.4}).  Consequently, it follows   that  by (\ref{eq1-thm2.4}),
\[
E[f^{anova}]=E[\hat{f}^{anova}]=E[\hat{f}^{cut}]+\sum_{i\in \{1,\cdots, N\}\setminus \mathcal{J}} E[f_i^{cut}]=E[f^{cut}].
\]
This proves the equality (\ref{mean-eq}). For ANOVA-HDMR, we note that
\begin{eqnarray}
\label{eq3-thm2.4}
\begin{split}
\mbox{Var}[f^{anova}]&=\|f^{anova}\|^2_{\mathcal{F}}=\|\hat{f}^{anova}\|_{\mathcal{F}}^2+\sum_{i\in \{1,\cdots, N\}\setminus \mathcal{J}}\|f_i^{anova}\|_{\mathcal{F}}^2\\
&=\mbox{Var}[\hat{f}^{anova}]+\sum_{i\in \{1,\cdots, N\}\setminus \mathcal{J}}\mbox{Var}[f_i^{anova}]\\
&=\mbox{Var}[\hat{f}^{cut}]+\sum_{i\in \{1,\cdots, N\}\setminus \mathcal{J}}\mbox{Var}[f_i^{cut}],
\end{split}
\end{eqnarray}
where  equations (\ref{eq1-thm2.4}) and (\ref{eq2-thm2.4}) have been used in the last step.
Further, by equations (\ref{eq1-thm2.4}) and (\ref{eq2-thm2.4}) we have
\begin{eqnarray}
\begin{split}
f^{cut}(\Theta)&=\hat{f}^{cut}(\Theta_{\mathcal{J}})+\sum_{i\in \{1,\cdots, N\}\setminus \mathcal{J}} f_i^{cut}(\theta_i)\\
&=\hat{f}^{anova}(\Theta_{\mathcal{J}})-\sum_{i\in \{1,\cdots, N\}\setminus \mathcal{J}} E[f_i^{cut}]
+ \sum_{i\in \{1,\cdots, N\}\setminus \mathcal{J}}\big(f_i^{anova}(\theta_i)-C_i\big)\\
&=f^{anova}(\Theta)-\bigg(\sum_{i\in \{1,\cdots, N\}\setminus \mathcal{J}} \big(E[f_i^{cut}]+C_i\big)\bigg).
\end{split}
\end{eqnarray}
This implies that $\mbox{Var}[f^{cut}]=\mbox{Var}[f^{anova}]$. Hence the proof is completed.

%%%%%%%%%%%%%%%%%%%%%%%%%%

\end{document}